\documentclass[11pt]{amsart}

\usepackage[foot]{amsaddr}
\usepackage{color,graphicx}

\usepackage{dsfont}
\usepackage{amssymb}
\usepackage{amsmath}
\usepackage{amsfonts}
\usepackage{graphicx}
\usepackage{amsthm}
\usepackage{enumerate}
\usepackage[mathscr]{eucal}
\usepackage{mathrsfs}
\usepackage{verbatim}
\usepackage{yhmath}

\calclayout

\usepackage{hyperref}
\hypersetup{colorlinks, linkcolor=blue}
\usepackage{soul}
\soulregister\cite7
\soulregister\eqref7
\soulregister\eqref7
\hypersetup{colorlinks, linkcolor=blue}

\newtheorem{thm}{Theorem}[section]
\newtheorem{lemma}[thm]{Lemma}
\newtheorem{prop}[thm]{Proposition}
\newtheorem{coro}[thm]{Corollary}

\theoremstyle{definition}
\newtheorem{definition}[thm]{Definition}

\newtheorem{remark}[thm]{Remark}

\newcommand\ddd{\mathrm{d}}

\newcommand\supp{\mathrm{supp}}

\newcommand\rhu{\rightharpoonup}
\newcommand\rsa{\rightsquigarrow}
\newcommand\bR{\mathbb{R}}

\newcommand\bZ{\mathbb{Z}}

\newcommand\bS{\mathbb{S}}

\def \l {\left}
\def \r {\right}

\makeatletter
\@namedef{subjclassname@2020}{%
  \textup{2020} Mathematics Subject Classification}
\makeatother

\allowdisplaybreaks[4]

\begin{document}

\title[Sharp Fourier extension on fractional surfaces]{Sharp Fourier extension on fractional surfaces}

\author{Boning Di$^{1,2}$}
\author{Dunyan Yan$^1$}
\address{$^1$ School of Mathematical Sciences, University of Chinese Academy of Sciences, Beijing, 100049, People's Republic of China}
\address{$^2$ Academy of Mathematics and Systems Science, Chinese Academy of Sciences, Beijing, 100190, People's Republic of China}
\email{diboning18@mails.ucas.ac.cn (ORCID: 0000-0002-2609-4645)}
\email{ydunyan@ucas.ac.cn (ORCID: 0000-0001-7462-3823)}

\date{} 
\thanks{This work was supported by the National Key R\&D Program of China [Grant No. 2023YFC3007303], National Natural Science Foundation of China [Grant Nos. 12271501 \& 12071052] and China Postdoctoral Science Foundation [Grant No. GZB20230812].}

\subjclass[2020]{Primary 42B10; Secondary 42B37, 35B38, 35Q41.}
\keywords{Sharp Fourier restriction theory, extremals, fractional Schr\"odinger equations, Strichartz inequalities.}
\begin{abstract}
We investigate a class of Fourier extension operators on fractional surfaces $(\xi,|\xi|^\alpha)$ with $\alpha\geq 2$. For the corresponding $\alpha$-Strichartz inequalities, we characterize the precompactness of extremal sequences by applying the missing mass method and bilinear restriction theory. Our result is valid in any dimension. In particular for dimension two, our result implies the existence of extremals for $\alpha \in [2,\alpha_0)$ with some $\alpha_0>5$.
\end{abstract}

\maketitle

\section{Introduction}
For $\alpha\geq 2$ and the corresponding $\alpha$-order free Schr\"odinger equation, the classical \textit{$\alpha$-Strichartz inequality} of \cite[Theorem 3.1]{KPV1991} states the following estimate
\begin{equation}\label{E:alpha-Strichartz}
	\l\|[D^{\frac{\alpha-2}{q}}] [e^{it|\nabla|^{\alpha}}] f \r\|_{L_t^q(\bR)L_x^r(\bR^d)} \leq \mathbf{M}_{d,q,\alpha} \|f\|_{L^2(\bR^d)},
\end{equation}
where $2/q+d/r=d/2$ with $q>2$ and
\[\mathbf{M}_{d,q,\alpha}:= \sup\l\{\l\|[D^{\frac{(\alpha-2)}{q}}] [e^{it|\nabla|^{\alpha}}]f \r\|_{L_t^q(\bR)L_x^r(\bR^d)}: \|f\|_{L_x^2(\bR^d)}=1\r\}\]
is the sharp constant, as well as
\[[e^{it|\nabla|^{\alpha}}]f(x):=\mathscr{F}^{-1}e^{it|\xi|^{\alpha}}\mathscr{F}f(x)= \frac{1}{(2\pi)^d} \int_{\bR^d} e^{ix\xi+it|\xi|^{\alpha}} \widehat{f}(\xi) \ddd \xi,\quad [D^s]u(x):=\mathscr{F}^{-1}|\xi|^s\mathscr{F}f(x).\]
Here $\mathscr{F}$ denotes the spatial Fourier transform
\[\mathscr{F}f(\xi):=\widehat{f}(\xi)=\int_{\bR^d} e^{-ix\xi}f(x)\ddd x, \quad x\xi:=x_1\xi_1+x_2\xi_2+\cdots+x_d\xi_d\]
for $x=(x_1,x_2,\ldots, x_d)$ and $\xi=(\xi_1,\xi_2,\ldots,\xi_d)$ in $\bR^d$. Indeed, this $\alpha$-Strichartz inequality \eqref{E:alpha-Strichartz} belongs to the wider class of \textit{Fourier extension estimates} since the space-time Fourier support of $[e^{it|\nabla|^{\alpha}}]f$ is on the \textit{fractional surface} $(\xi,|\xi|^{\alpha}) \subset \bR^{d+1}$. For convenience, we denote that
\[[E_{\alpha}]f(t,x):=[D^{\frac{\alpha-2}{q_0}}][e^{it|\nabla|^{\alpha}}]f(x), \quad q_0:=\frac{2d+4}{d}, \quad \mathbf{M}_{d,\alpha}:= \mathbf{M}_{d,q_0,\alpha}, \quad \mathbf{S}_d^{*}:=\mathbf{M}_{d,2}.\]
Note that $\mathbf{S}_d^{*}$ is the corresponding sharp constant for the classical Schr\"odinger operator $[e^{it\Delta}]$ and the case $\alpha=2$ is also known as the \textit{Stein-Tomas estimate for paraboloid}.

The relevant \textit{symmetries} for these $\alpha$-Strichartz inequalities are the space-time translations and scaling as follows
\[[g_n]f(x):=(h_n)^{d/2} [e^{it_n|\nabla|^{\alpha}}]f (h_nx +x_n), \quad (h_n,x_n,t_n)\in \bR_{+}\times \bR^d \times\bR;\]
and the \textit{associated group} $G$ is defined by
\[G:=\Big\{[g_n]: (h_n,x_n,t_n)\in \bR_{+}\times \bR^d \times\bR\Big\}.\]
Then we say a sequence of functions $(f_n)$ in $L^2(\bR^d)$ is \textit{precompact up to symmetries} if there exists a sequence of symmetries $([g_n])$ in $G$ such that $\l([g_n]f_n\r)$ has convergent subsequence in $L^2(\bR^d)$. Meanwhile, a sequence of functions $(f_n)$ in $L^2(\bR^d)$ is an \textit{extremal sequence} for $\mathbf{M}_{d,q,\alpha}$ if it satisfies
\[\|f_n\|_{L^2(\bR^d)}=1,\quad \lim_{n\to\infty} \l\|[D^{\frac{\alpha-2}{q_0}}] [e^{it|\nabla|^{\alpha}}] f_n\r\|_{L_{t,x}^{q_0}(\bR^{d+1})}=\mathbf{M}_{d,q,\alpha}.\]
Furthermore, a function $f_* \in L^2(\bR^d)$ is called an \textit{extremal} for $\mathbf{M}_{d,q,\alpha}$ if $f_*$ can make the inequality \eqref{E:alpha-Strichartz} an equality and $\|f_{*}\|_{L^2}=1$.

The sharp Fourier restriction theory, equivalently the extremal problems for Strichartz type inequalities, has received much attention recently. Readers are referred to the survey \cite{FO2017} and the references therein for some progress on this theory, see also the recent survey \cite{NOT2022}. We sketch briefly some of the works as follows.
\begin{itemize}
	\item \textbf{Paraboloid:} Kunze \cite{Kunze2003} showed the existence of extremals for the Fourier extension operator on one-dimension parabola $(\xi,\xi^2) \subset \bR^2$  based on an application of \textit{concentration-compactness principle} from Lions \cite{Lions1984a,Lions1984b}; then Foschi \cite{Foschi2007} proved that the only extremals are Gaussians for the case of one-dimension and two-dimension paraboloids by solving some functional equations and investigating some Cauchy-Schwarz inequalities; meanwhile, Hundertmark and Zharnitsky \cite{HZ2006} established the same result independently by giving a new representation for Strichartz integral based on some orthogonal projection operators; later, Shao \cite{Shao2009EJDE} showed the existence of extremals for the case of arbitrary dimensional paraboloids by applying the \textit{profile decomposition} consequence of \cite{BV2007}.
	\item \textbf{Cone:} for the case of $(\xi,|\xi|) \subset \bR^{d+1}$ with low dimensions $d=\{1,2,3\}$, the only extremals are known to be exponentials by the works of Foschi \cite{Foschi2007} and Carneiro \cite{Carneiro2009}; for the case of higher dimensions $d\geq4$, the extremals exist due to the work of Ramos \cite{Ramos2012}.
	\item \textbf{Sphere:} Christ and Shao \cite{CS2012A&P} showed the existence of extremals for the Fourier extension operator on the two-dimension sphere $\bS^2$ by following the general concentration compactness framework, as well as establishing some strict comparisons for the sharp constants of sphere and paraboloid; then, for this $\bS^2$ case, Foschi \cite{Foschi2015} proved that the only extremals are constants by investigating the Cauchy-Schwarz type estimates for some quadratic forms based on the geometric feature of $\bS^2$; later, Shao \cite{Shao2016} obtained the existence of extremals for the one-dimension sphere $\bS^1$ by combining the outlines in \cite{CS2012A&P} and the profile decomposition ideas in \cite{BG1999,CK2007}; then for arbitrary dimensions, Frank, Lieb and Sabin \cite{FLS2016} established a characterization for the precompactness of extremal sequences by applying the \textit{missing mass method} from Lieb \cite{Lieb1983}.
	\item \textbf{Other situations:} There are many related works such as the odd curves \cite{BOQ2020,FS2018,Shao2009}, hyperboloids \cite{COS2019ANIHPC,COSS2021}, perturbations \cite{OQ2018}, and non-endpoint type estimates \cite{FVV2011,FVV2012,HS2012}, as well as $L^p$ extremals \cite{BS2023,CQ2014,FS2022,Stovall2020}.
\end{itemize}

The natural generalization of the paraboloid case is to investigate the sharp Fourier extension on fractional surfaces $(\xi,|\xi|^{\alpha}) \subset \bR^{d+1}$, which is corresponding to the fractional Schr\"odinger equations. For the case $(d,\alpha)=(1,4)$, Jiang, Pausader and Shao \cite{JPS2010} established a dichotomy result on the existence of extremals by establishing the corresponding linear profile decomposition for one-dimension forth order Schr\"odinger equations; for the case $(d,\alpha)=(d,4)$, Jiang, Shao and Stovall \cite{JSS2017} studied the high-dimension forth order Schr\"odinger equations and established a dichotomy result on the existence of extremals; then for the case $(d,\alpha)=(2,4)$, Oliveira e Silva and Quilodr\'an \cite{OQ2018} resolved this dichotomy and obtained the existence of extremals by applying some comparison principle for convolutions of certain singular measures; later, Brocchi, Oliveira e Silva and Quilodr\'an \cite{BOQ2020} established a dichotomy result for the case of $(d,\alpha)=(1,\alpha)$ by following some concentration compactness arguments and further obtained the existence of extremals for all $\alpha\in (1,\alpha_0)$ with some $\alpha_0>5$ by applying the aforementioned comparison principle; recently, the authors \cite{DY2023} established same dichotomy result for the case of $(d,\alpha)=(1,\alpha)$ by establishing the corresponding linear profile decomposition for one-dimension fractional Schr\"odinger equations, and then further studied the asymmetric as well as non-endpoint Strichartz inequalities. For the case of one-dimension fractional curves, both of the proofs in \cite{BOQ2020} and \cite{DY2023} are based on some refined Strichartz estimates, which follows from the Hausdorff-Young inequality and Whitney decomposition. However, these techniques cannot deal with the higher-dimension fractional surfaces case.

In this article, we investigate the general $(d,\alpha)$ case. One of our main results is the following existence of extremals consequence Theorem \ref{T:Existence}. This result generalizes the aforementioned result of \cite[Theorem 1.6]{OQ2018} which claims the existence of extremals for the case of $(d,\alpha)=(2,4)$.
\begin{thm} \label{T:Existence}
	For dimension $d=2$, there exists one constant $\alpha_0>5$ such that for arbitrary $\alpha\in[2,\alpha_0)$ the extremal for $\mathbf{M}_{2,\alpha}$ exists.
\end{thm}
To prove this result, we need the following precompactness Theorem \ref{T:Precompactness}, which gives one characterization for the precompactness of extremal sequences. With this precompactness theorem in place, as we will show later in Section \ref{S:Existence of extremals}, our Theorem \ref{T:Existence} follows directly from the previous results in \cite[Proposition 6.9]{OQ2018} and the classical fact that Gaussians are extremals for $\mathbf{S}_2^{*}$.
\begin{thm}\label{T:Precompactness}
	All extremal sequences for $\mathbf{M}_{d,\alpha}$ are precompact up to symmetries if and only if
	\begin{equation}\label{T:Precompactness-1}
		\mathbf{M}_{d,\alpha}>(\alpha-1)^{\frac{-1}{2d+4}} (\alpha/2)^{\frac{-d}{2d+4}} \mathbf{S}_d^{*}.
	\end{equation}
	In particular, if the strict inequality \eqref{T:Precompactness-1} holds, then there exists an extremal for $\mathbf{M}_{d,\alpha}$.
\end{thm}
It is obvious that the Strichartz norm is invariant under the actions of aforementioned symmetries\footnote{Indeed, the Strichartz estimates are also invariant under some other transformations such as the rotation symmetries $f(x)\mapsto e^{ix_0}f(x)$. However, we will not use these symmetries here since they do not lead to loss of compactness in $L^2(\bR^d)$ and are inessential in our situation. Furthermore, for the special case $\alpha=2$, there are also frequency-translation symmetries $f(x)\mapsto e^{ix\xi_0}f(x)$ which does not maintain the Strichartz norm for general $\alpha$.}, hence precompact up to symmetries is the best one can expect. Note that Theorem \ref{T:Precompactness} states some universal property for all extremal sequences instead of identifying the extremals. Similar consequences are also established in previous literature, such as \cite[Theorem 1.1]{FLS2016} for the sphere and \cite[Theorem 1]{FS2018} for cubic curve. As mentioned above, the one-dimension case of Theorem \ref{T:Precompactness} has been proved in the recent works \cite[Theorem 1.3]{BOQ2020} and \cite[Theorem 1.1]{DY2023} by different methods. For the sharp constant $\mathbf{S}_d^{*}$, it is conjectured in \cite{Foschi2007,HZ2006} that the only extremals are Gaussians and then the corresponding constant can be obtained by the residue theorem. It can be seen from the asymptotic Schr\"{o}dinger Lemma \ref{L:Asymptotic schrodinger} that $(\alpha-1)^{\frac{-1}{2d+4}} (\alpha/2)^{\frac{-d}{2d+4}} \mathbf{S}_d^{*}$ is a lower bound for $\mathbf{M}_{d,\alpha}$. In addition, the strict inequality \eqref{T:Precompactness-1} has been proved for some special cases and further details are discussed in Remark \ref{R:Asymptotic schrodinger}.

Now we give some remarks on the proof of Theorem \ref{T:Precompactness}. Recall that the Banach-Alaoglu theorem implies that all extremal sequences must have weak limits up to subsequences. Hence the precompactness of extremal sequences (as well as the existence of extremals) usually comes from two steps: finding a nonzero weak limit and then upgrading this weak convergence to strong convergence. The first step often relies on some refinement of the original $\alpha$-Strichartz estimates \eqref{E:alpha-Strichartz}, and the second step often follows from some compactness arguments. In this paper, we achieve these two steps by using the bilinear restriction estimates and the missing mass method separately.

The fact that Tao's bilinear restriction estimates \cite{Tao2003} could deduce some refined Strichartz estimates is first shown by B\'egout and Vargas for the classical Schr\"odinger equations in \cite{BV2007}, which generalizes the previous low-dimensional results in \cite{KPV1991,Keraani2001,MV1998} to higher dimensions. And then numerous consequences are established by following this idea, see, for instance, \cite{COSS2021,FLS2016,KV2013}. In fact, it is also mentioned in \cite[Section 2]{BOQ2020} and \cite[Section 3.3]{FS2018} that the refined Strichartz-type estimates usually come from the bilinear restriction estimates. However, both of them can simply use the Hausdorff-Young inequality instead, since they both only study the one dimensional fractional curves. This Hausdorff-Young inequality is also used by the authors in the recent work \cite{DY2023}. Inspired by these previous results, in this paper we will use Tao's bilinear restriction estimates to establish the desired high dimensional refined $\alpha$-Strichartz estimates Proposition \ref{P:Refined alpha-Strichartz} for fractional surfaces and then use this result to establish the non-zero weak limit.

However, there are two potential difficulties we should resolve: one is that the fractional surface $(\xi, |\xi|^{\alpha})$ has zero Gaussian curvature at the origin point, which means that we cannot use Tao's bilinear restriction estimate \cite[Section 9]{Tao2003} directly; another one is that (except for the origin point) the geometric structure of fractional surface is different from that of paraboloid, which implies that we need to investigate the corresponding quasi-orthogonality of these fractional surfaces in order to apply the bilinear-to-linear arguments in \cite{TVV1998}. As shown later in Section \ref{S:Bilinear restriction and refined Strichartz}, we settle the first difficulty by applying some annular orthogonality consequence Lemma \ref{L:Annular orthogonality} and then restrict our attention to the annular case, which are inspired by \cite{COSS2021,JSS2017,KSV2012}; furthermore, based on one geometric result Proposition \ref{P:Angle decomposition}, the second difficulty can be solved by dividing the angle into a large number of regions (see the number $K_{d,\alpha}$ in Section \ref{S:Bilinear restriction and refined Strichartz}) and then investigating the quasi-orthogonality of each regions.

To upgrade the weak convergence to strong convergence, we use the missing mass method which is invented by Lieb \cite{Lieb1983} in the content of Hardy-Littlewood-Sobolev inequality; see, for instance, \cite{FLS2016,FS2018} for the applications of this method in sharp Fourier restriction theory lately. One crucial tool to apply this method is the Br\'ezis-Lieb type lemma due to \cite{BL1983,Lieb1983}, and here we use a more general version \cite[Lemma 3.1]{FLS2016}. There are also various other kinds of generalizations appeared in the literature such as \cite{BOQ2020,COS2019ANIHPC,FVV2011}. For a sequence of functions, when we decompose each function into different parts, these Br\'ezis-Lieb type lemmas can give some limit-orthogonality properties under some suitable conditions. More specifically, the main required condition is a pointwise convergence assumption which in turn relies on the corresponding local smoothing estimates. Furthermore, in high dimensions, there need some multi-variable analysis such as the multi-variable Taylor's theorem and decay estimates for multi-variable oscillatory integrals.

Let us roughly explain how our strategy works. It seems that this also obeys Lions' concentration-compactness principle \cite{Lions1984a,Lions1984b}. Note that the loss of compactness in $L^2(\bR^d)$ is the main enemy in our situation. Also recall that there are some symmetries when we investigate the precompactness of extremal sequences for $\mathbf{M}_{d,\alpha}$. These three parameters $(h_n,t_n,x_n)$ in the symmetries represent three possible ways to lose compactness in $L^2(\bR^d)$: scaling and space-time translations. Fortunately, by using the terminology ``up to symmetries" to eliminate the effect of these three parameters, we do not need to worry about the aforementioned three ways of losing compactness. However, there is obvious another way to lose compactness in $L^2(\bR^d)$: the frequency translations $f(x)\mapsto e^{ix\xi_n}f(x)$ with parameters $|\xi_n|\to\infty$. Notice that the Strichartz norm changes when frequency translation occurs. Therefore, to establish the precompactness of extremal sequences (up to symmetries), we should at least understand the effect of frequency translations and exclude this type of possibility of losing compactness. Then we are led to the asymptotic Schr\"odinger behavior Lemma \ref{L:Asymptotic schrodinger} and this is why the strict inequality \eqref{T:Precompactness-1} appears in our Theorem \ref{T:Precompactness}: to exclude the loss of compactness deduced by frequency translations. In this sense, our Theorem \ref{T:Precompactness} states one fact that essentially the aforementioned four ways are the only ways of losing compactness when investigating the extremal sequences of $\mathbf{M}_{d,\alpha}$.

The outline of this paper is as follows. In Section \ref{S:Bilinear restriction and refined Strichartz}, we use the bilinear restriction theory to deduce the refined $\alpha$-Strichartz estimates as well as the non-zero weak limit consequence; meanwhile, our arguments in this section rely on an auxiliary geometric result whose proof is postponed to the Appendix \ref{S:A geometric result}. Then in Section \ref{S:Local smoothing and local convergence}, we present some pointwise convergence results so that we are able to apply the Br\'ezis-Lieb type lemma. Next in Section \ref{S:Approximate symmetry}, we study the effect of frequency parameters including the asymptotic Schr\"odinger behavior and the corresponding pointwise convergence property. In Section \ref{S:Method of Missing Mass}, we apply the missing mass method to establish our precompactness Theorem \ref{T:Precompactness}. Finally in Section \ref{S:Existence of extremals}, we show our existence of extremals result Theorem \ref{T:Existence}.

We end this section with some notations. The familiar notation $x\lesssim y$ denotes that there exists a finite constant $C$ such that $|x|\leq C|y|$, similarly for $x\gtrsim y$ and $x\sim y$. If necessary, we may use the notation $x\lesssim_{\alpha} y$ to show the dependence of this aforementioned constant $C=C_{\alpha}=C(\alpha)$. Finally the indicator function of a set $E$ will be denoted by $\mathds{1}_E$ and we further define $\widehat{f}_E:=\mathds{1}_E\widehat{f}$.

\section{Bilinear restriction and refined Strichartz}\label{S:Bilinear restriction and refined Strichartz}
In this section, we apply the bilinear restriction theory to establish our desired high dimensional refined $\alpha$-Strichartz estimates, see Proposition \ref{P:Refined alpha-Strichartz} below. To achieve this, we need an auxiliary function and a relevant geometric result which is provided in Appendix \ref{S:A geometric result}. For a vector $\xi=(\xi_1,\xi_2,\cdots,\xi_d) \in\bR^d$, we introduce the notations
\[\|\xi\|_{\max}:=\max\{|\xi_1|, |\xi_2|, \ldots, |\xi_d|\}, \quad \|\xi\|_{\min}:=\min\{|\xi_1|, |\xi_2|, \ldots, |\xi_d|\}.\]
Let $\mathcal{D}$ denote the family of all the classical dyadic cubes in $\bR^d$. We further define the \textit{annular dyadic cubes} as follows.
\begin{definition}
	For $N\in 2^{\bZ}$, we define the \textit{dyadic cube-annular $\mathcal{A}_N$ on Fourier space} as
	\[\mathcal{A}_N:=\l\{\xi \in\bR^d: N\leq \|\xi\|_{\max}<2N\r\}.\]
	Moreover, for $r\in 2^{\bZ_{+}}$, we decompose $\mathcal{A}_N$ into \textit{annular dyadic cubes} which are
	\[\mathcal{D}_{N,r}:=\l\{\tau\in \mathcal{D}: \tau \subset \mathcal{A}_N, \ell(\tau)=\frac{N}{r}\r\}.\]
\end{definition}
To apply some Whitney-type decomposition and achieve some quasi-orthogonal properties, we need the following preliminary definitions and notations.
\begin{definition}
	For dyadic cubes $\tau\subset \mathcal{A}_N$ and $\tau'\subset \mathcal{A}_N$, we write $\tau\sim\tau'\subset \mathcal{A}_N$ if $\ell(\tau)=\ell(\tau')$ and they are not adjacent, their parents are not adjacent, their $2$-parents are not adjacent, ..., their $(N_{d,\alpha}-1)$-parents are not adjacent while their $N_{d,\alpha}$-parents are adjacent. Here we say two regions are \textit{adjacent} if their closures intersect, and the number $N_{d,\alpha}\in\bZ_{+}$ will be determined later in \eqref{E:Angle condition}.
\end{definition}
Notice that $\tau\sim\tau'$ will imply $\ell(\tau)=N/r$ with some $r\geq 2^{N_{d,\alpha}}$. To use the bilinear form and the Whitney-type decomposition, we divide the unit sphere $\bS^{d-1}$ into $K_{d,\alpha}$ parts which will lead to the following angle decomposition
\[\mathcal{A}_N=\bigcup_{j=1}^{K_{d,\alpha}}\mathcal{A}_N^{j},\quad \theta_N^j:=\sup\l\{\theta(\xi,\eta): (\xi,\eta)\in \mathcal{A}_N^j\times \mathcal{A}_N^j\r\},\]
such that $\theta_N^j <\theta_{d,\alpha}$ with the angle range $\theta_{d,\alpha}$ depending on $K_{d,\alpha}$. Here $\theta(\xi,\eta)$ denotes the angle between the vectors $\xi$ and $\eta$. By symmetry we may assume $|\xi|\geq |\eta|$, meanwhile we can set $K_{d,\alpha}$ large enough to let the angle region $\theta_{d,\alpha}$ as small as we want. Here the numbers $N_{d,\alpha}$ and $K_{d,\alpha}$ are chosen based on the angle decomposition result Proposition \ref{P:Angle decomposition}. Indeed, using the notations in Remark \ref{R:Angle decomposition}, our numbers $N_{d,\alpha}$ and $K_{d,\alpha}$ are chosen such that
\begin{equation}\label{E:Angle condition}
	\arctan (d^{-1/2}2^{-N_{d,\alpha}}) + \theta_{d,\alpha} \leq \bar{\theta}_0.
\end{equation}
Since the angle $\bar{\theta}_0$ is fixed and depends only on $(d,\alpha)$, we know that this condition \eqref{E:Angle condition} must can be achieved as long as $N_{d,\alpha}$ large enough and $\theta_{d,\alpha}$ small enough which means $K_{d,\alpha}$ large enough.

The reason we construct this condition \eqref{E:Angle condition} is that it will imply the result \eqref{E:Quasi-orthogonality-7} which is a critical estimate in the proof of Quasi-orthogonality Lemma \ref{L:Quasi-orthogonality}. Meanwhile it should be pointed out that
\begin{equation}\label{E:Angular divided}
	\max_{1\leq j\leq K_{d,\alpha}} \l\|[e^{it|\nabla|^{\alpha}}] f_{\mathcal{A}_{N}^j} \r\|_{L_{t,x}^{q_0}} \leq \l\|[e^{it|\nabla|^{\alpha}}] f_{\mathcal{A}_N}\r\|_{L_{t,x}^{q_0}} \leq \sum_{j=1}^{K_{d,\alpha}} \l\|[e^{it|\nabla|^{\alpha}}] f_{\mathcal{A}_{N}^j} \r\|_{L_{t,x}^{q_0}}.
\end{equation}
We further introduce the \textit{annular-restricted notation} $\tau\sim\tau'\subset \mathcal{A}_N^j$ which means
\[\tau\sim\tau'\subset \mathcal{A}_N, \quad \tau\cap\mathcal{A}_N^j\neq \emptyset, \quad \tau'\cap\mathcal{A}_N^j \neq \emptyset.\]
Recall that $\bar{\theta}_0$ is very small and much smaller than $\pi/8$. Hence by dividing $\mathcal{A}_N$ into $K_{d,\alpha}$ parts as above, except for a null set, it can be achieved that for arbitrary $(\xi,\xi') \in \mathcal{A}_{N}^{j}\times \mathcal{A}_N^j$ there exists unique pair $\tau\sim \tau'\subset \mathcal{A}_N^j$ satisfying $\xi\in \tau$ and $\xi'\in \tau'$. This fact deduces the Whitney-type decomposition which will be used later in the proof of annular refined estimates Lemma \ref{L:Annular refined alpha-Strichartz}.

In our bilinear setting, the first crucial fact is the following quasi-orthogonality lemma.

\begin{lemma}\label{L:Quasi-orthogonality}
	Suppose that $d\geq 2$ and $j_0\in \{1,2\ldots, K_{d,\alpha}\}$. Then the following inequality holds
	\[\l\|\sum_{\tau\sim\tau' \subset \mathcal{A}_1^{j_0}} [e^{it|\nabla|^{\alpha}}] f_{\tau} \cdot [e^{it|\nabla|^{\alpha}}] f_{\tau'}\r\|_{L_{t,x}^{\frac{q_0}{2}}}^{\frac{q_0}{2}} \lesssim \sum_{\tau\sim\tau' \subset \mathcal{A}_1^{j_0}} \l\|[e^{it|\nabla|^{\alpha}}] f_{\tau} \cdot [e^{it|\nabla|^{\alpha}}] f_{\tau'}\r\|_{L_{t,x}^{\frac{q_0}{2}}}^{\frac{q_0}{2}},\]
	for all $f\in L^2(\bR^d)$ satisfying $\supp{\widehat{f}}\subset \mathcal{A}_1^{j_0}$.
\end{lemma}
\begin{proof}[\textbf{Proof of Lemma \ref{L:Quasi-orthogonality}}]
	At the beginning, we mention that all the dyadic cubes in this proof have intersections with $\mathcal{A}_1^{j_0}$. Given $\tau\in \mathcal{D}_{1,r}$, denote $\xi_0:=c(\tau)$ the center of $\tau$. Then for every $\xi\in \tau$, considering the radial direction and angle differences, some geometry observations give that
	\begin{equation}\label{E:Quasi-orthogonality-1}
		\Big||\xi|-|\xi_0|\Big|\lesssim 1/r, \quad \l(|\xi||\xi_0|-\xi\xi_0\r)^{1/2}\lesssim 1/r.
	\end{equation}
	Furthermore for $\tau\sim\tau' \in \mathcal{D}_{1,r}$ and $\xi_0':=c(\tau')$, there holds
	\begin{equation}\label{E:Quasi-orthogonality-2}
		\Big||\xi_0|-|\xi'_0|\Big| +(|\xi_0||\xi'_0|-\xi_0\xi'_0)^{1/2} \sim 1/r.
	\end{equation}
	Denote $\tilde{\tau}, \tilde{\tau}'$ the lifts of $\tau, \tau'$ into the surface $(\xi, |\xi|^{\alpha})\subset \bR^{d+1}$. Based on the quasi-orthogonality result \cite[Lemma 2.2]{Ramos2012}, our main task is investigating the geometry of the sum-set
	\[\tilde{\tau}+\tilde{\tau}':=\l\{(\xi+\xi', |\xi|^{\alpha}+|\xi'|^{\alpha}): (\xi,\xi')\in \tau\times \tau'\r\} \subset \bR^{d+1}.\]
	We claim that the collection $\{(\tau,\tau'):\tau\sim \tau'\subset \mathcal{A}_1^{j_0}\}$ can be finitely decomposed into universal number of subsets $\l\{\mathscr{T}_m\r\}_{m=1}^{K_0}$, and then for each subset $\mathscr{T}_m$ there exist a universal number $K_1$ and a universal constant $\beta>0$ such that the following property holds: for every $(\tau,\tau') \in \mathscr{T}_m$,  we can find $K_1$ parallelepipeds $\{P_{\ell}=P_{\ell}(\tau,\tau')\}_{\ell=1}^{K_1}$ satisfying
	\[\l(\tilde{\tau}+\tilde{\tau}'\r) \subset \l(\bigcup_{\ell=1}^{K_1} P_{\ell}\r), \quad\quad \Big[(1+\beta)\cdot P_{\ell}\Big] \cap \Big[(1+\beta)\cdot P_{\ell'}\Big]=\emptyset ~ \text{for} ~\ell\neq \ell'.\]
	Here $(1+\beta)\cdot P_{\ell}$ denotes the centered dilation of $P_{\ell}$ which means
	\[(1+\beta)\cdot P_{\ell}:=(1+\beta)[P_{\ell}-c(P_{\ell})]+c(P_{\ell})\]
	with $c(P_{\ell})$ denoting the center of $P_{\ell}$. Let us postpone the detailed proof of this claim and use it to prove our final result now. Define $\Psi_{\ell}:=\mathds{1}_{P_{\ell}}$. Since $P_{\ell}$ is a parallelepiped, for every exponent $q>1$, the boundedness of Hilbert transform implies the following multiplier boundedness
	\begin{equation}\label{E:Quasi-orthogonality-3}
		\|F\ast \widehat{\Psi}_{\ell}\|_{L^q(\bR^{d+1})} \lesssim\|F\|_{L^q(\bR^{d+1})}
	\end{equation}
	for all $F\in L^q(\bR^{d+1})$. On the other hand, note that the space-time Fourier support satisfies
	\[\supp \Big([e^{it|\nabla|^{\alpha}}]f_{\tau}[e^{it|\nabla|^{\alpha}}]f_{\tau'}\Big)^{\bigwedge_{t,x}} \subset \l(\tilde{\tau}+\tilde{\tau}'\r).\]
	Hence by the claim and triangle inequality, we only need to show the following estimate holds
	\begin{equation}\label{E:Quasi-orthogonality-4}
		\l\|\sum_{(\tau,\tau')\in \mathscr{T}_m} \l([e^{it|\nabla|^{\alpha}}]f_{\tau}[e^{it|\nabla|^{\alpha}}]f_{\tau'}\r) \ast \widehat{\Psi}_{\ell}\r\|_{L_{t,x}^{\frac{q_0}{2}}}^{\frac{q_0}{2}} \lesssim \sum_{(\tau,\tau')\in \mathscr{T}_m} \l\| [e^{it|\nabla|^{\alpha}}]f_{\tau}[e^{it|\nabla|^{\alpha}}]f_{\tau'}\r\|_{L_{t,x}^{\frac{q_0}{2}}}^{\frac{q_0}{2}}
	\end{equation}
	for each $m\in \{1,2,\ldots, K_0\}$ and each $\ell\in \{1,2,\ldots, K_1\}$. This estimate \eqref{E:Quasi-orthogonality-4} comes from a direct application of \cite[Lemma 2.2]{Ramos2012} and the boundedness \eqref{E:Quasi-orthogonality-3}. Indeed, since $P_{\ell}$ is an affine image of the unit cube, for each $(\tau,\tau')\in\mathscr{T}_m$ we can construct a bump function $\varphi=\varphi(\tau,\tau')$ satisfying
	\[\supp \varphi \subset (1+\beta)\cdot P_{\ell}, \quad \varphi(x)\equiv 1 ~\text{for}~ x\in P_{\ell}, \quad \|\widehat{\varphi}\|_{L^1(\bR^{d+1})}\leq C,\]
	where the constant $C$ is independent of $(\tau,\tau')$. Thus, \cite[Lemma 2.2]{Ramos2012} deduces that
	\[\l\|\sum_{(\tau,\tau')\in \mathscr{T}_m} \l([e^{it|\nabla|^{\alpha}}]f_{\tau}[e^{it|\nabla|^{\alpha}}]f_{\tau'}\r) \ast \widehat{\Psi}_{\ell}\r\|_{L_{t,x}^{\frac{q_0}{2}}}^{\frac{q_0}{2}} \lesssim \sum_{(\tau,\tau')\in \mathscr{T}_m} \l\| \l([e^{it|\nabla|^{\alpha}}]f_{\tau}[e^{it|\nabla|^{\alpha}}]f_{\tau'}\r) \ast \widehat{\Psi}_{\ell}\r\|_{L_{t,x}^{\frac{q_0}{2}}}^{\frac{q_0}{2}}.\]
	This conclusion and the boundedness \eqref{E:Quasi-orthogonality-3} immediately give the desired estimate \eqref{E:Quasi-orthogonality-4}.
	
	It remains to prove the aforementioned claim. For the sum-set $\tau+\tau'$, considering the length along radial direction and the angle differences, it is not hard to see that
	\begin{equation}\label{E:Quasi-orthogonality-5}
		\Big||\xi+\xi'|-|\xi_0+\xi_0'|\Big|\lesssim 1/r, \quad \l[|\xi+\xi'||\xi_0+\xi_0'|-(\xi+\xi')(\xi_0+\xi_0')\r]^{1/2}\lesssim 1/r.
	\end{equation}
	Meanwhile some further investigation will imply the following estimate
	\begin{equation}\label{E:Quasi-orthogonality-6}
		|\xi|^{\alpha}+|\xi'|^{\alpha}-|\xi+\xi'|^{\alpha}/2^{\alpha-1} \sim {1}/{r^2}.
	\end{equation}
	Indeed, without loss of generality we may assume $|\xi|\geq |\xi'|$. If we define the auxiliary function
	\[F_{\xi'}(\xi):=|\xi|^{\alpha}+|\xi'|^{\alpha}-|\xi+\xi'|^{\alpha}/2^{\alpha-1},\]
	then for every fixed $\xi'\in\bR^d$ the function $F_{\xi'}(\xi)$ has non-degenerate critical point $\xi=\xi'$. In other words, the Hessian matrix $\mathrm{Hess} F_{\xi'}(\xi')$ is positive-definite and
	\[\nabla F_{\xi'}(\xi')=0.\]
	Also notice that $F_{\xi'}(\xi')=0$ at this point and $F_{\xi'}(\xi)$ is a smooth function. Moreover, since we have
	\[(\xi,\xi')\in\tau\times \tau', \quad \tau\sim\tau'\subset \mathcal{A}_1^{j_0},\]
	the condition \eqref{E:Angle condition} and Proposition \ref{P:Angle decomposition} then imply
	\begin{equation}\label{E:Quasi-orthogonality-7}
		\det \mathrm{Hess}F_{\xi'}(\xi) \sim 1.
	\end{equation}
	Hence when $\tau\sim\tau'\in \mathcal{D}_{1,r}$ which leads to $|\xi-\xi'|\sim 1/r$, we can use the multi-variable Taylor's theorem to obtain $|F_{\xi'}(\xi)|\sim 1/r^2$. This gives the desired estimate \eqref{E:Quasi-orthogonality-6}.
	
	Then we follow some similar arguments in the proof of \cite[Lemma 5.2]{COSS2021}. First, we can find a universal constant $C_1$ such that for every fixed $\tau\sim\tau'\in\mathcal{D}_{1,r}$, the number of corresponding $\rho\sim \rho'\in \mathcal{D}_{1,s}$ satisfies the following quantity bound
	\[\#\l\{(\rho, \rho'): \rho\sim\rho'\in \mathcal{D}_{1,s}, ~(\tilde{\tau}+\tilde{\tau}')\cap (\tilde{\rho}+\tilde{\rho}')\neq \emptyset\r\} \leq C_1.\]
	In fact, if $(\tilde{\tau}+\tilde{\tau}')\cap (\tilde{\rho}+\tilde{\rho}')\neq \emptyset$, it is not hard to see that the estimate \eqref{E:Quasi-orthogonality-6} implies $s\sim r$; and for each fixed $s$, the estimates \eqref{E:Quasi-orthogonality-1}, \eqref{E:Quasi-orthogonality-2} and \eqref{E:Quasi-orthogonality-5} can imply that for given $\tau$, the number of possible cubes $\rho\in\mathcal{D}_{1,s}$ is universally bounded; then for each fixed $\rho$, the number of $\rho'$ is obviously uniformly finite. Therefore we obtain the existence of the constant $C_1$.
	
	Second, note that \eqref{E:Quasi-orthogonality-6} gives the universal constant $c_2>0$ and $c_3>0$ with the following relation
	\begin{equation*}
		\tilde{\tau}+\tilde{\tau}' \subset \Gamma_{\tau, \tau'}, \quad \Gamma_{\tau, \tau'}:=\l\{(\xi, \eta)\in (\tau+\tau')\times \bR: |\xi|^{\alpha}/2^{\alpha-1}+c_2/r^2\leq \eta \leq |\xi|^{\alpha}/2^{\alpha-1}+c_3/r^2\r\}.
	\end{equation*}
	Meanwhile the estimates \eqref{E:Quasi-orthogonality-5} give the rectangle $R_{\tau,\tau'}\subset \bR^d$ such that $\tau+\tau'\subset R_{\tau,\tau'}$ with center $c(R_{\tau,\tau'})$ at the point $\gamma_0:=\xi_0+\xi_0'$ and every side length comparable to $1/r$, as well as one edge aligned with the vector $\gamma_0$. By a centered dilation $R_{\tau,\tau'}^{*}:=(1+c_4)\cdot R_{\tau,\tau'}$ with constant $c_4>0$ small enough independent of $(\tau,\tau')$, we can let the following sets still have bounded overlap
	\[\Sigma_{\tau, \tau'}:=\l\{(\xi, \eta)\in R_{\tau, \tau'}^{*}\times \bR: |\xi|^{\alpha}/2^{\alpha-1}+c_2/2r^2\leq \eta \leq |\xi|^{\alpha}/2^{\alpha-1}+2c_3/r^2\r\}.\]
	Hence we are able to decompose the collection $\{(\tau,\tau'):\tau\sim \tau'\subset \mathcal{A}_1^{j_0}\}$ into universal number of subsets $\l\{\mathscr{T}_m\r\}_{m=1}^{K_0}$, such that for each $\mathscr{T}_m$ there is a corresponding set $\{\Sigma_{\tau, \tau'}\}_m$ whose elements are pairwise disjoint. Thus, we may fix some $m=m_0$ and investigate one subset $\mathscr{T}_{m_0}$ from now on.
	
	Third, for $\gamma\in \bR^d$, we define $T(\gamma)$ to be the tangent plane of the surface $(\xi, |\xi|^{\alpha}/2^{\alpha-1})$ at the point $(\gamma, |\gamma|^{\alpha}/2^{\alpha-1})$ as follows
	\[T(\gamma):=\l\{(\gamma,|\gamma|^{\alpha}/2^{\alpha-1})+v: v\in\bR^{d+1}, ~v\bot (\alpha|\gamma|^{\alpha-2}\gamma/2^{\alpha-1},-1)\r\}.\]
	Let $(e_1,e_2,\ldots, e_{d+1})$ denote the canonical basis in $\bR^{d+1}$. Without loss of generality, we may assume the center $\gamma_0=|\gamma_0|e_d$. Consider the point $\gamma=ke_d$ which is very close to $\gamma_0$. In this case, the corresponding hyperplane can be computed as
	\[T(\gamma)=\l\{(\gamma,|\gamma|^{\alpha}/2^{\alpha-1})+(v_1, v_2,\ldots, v_d, v_d\alpha k^{\alpha-1}/2^{\alpha-1}): v_i\in\bR, ~i=1,2,\ldots, d\r\};\]
	and lifting the rectangle $R_{\tau,\tau'}$ to the tangent plane $T(\gamma)$ amounts to choosing
	\[|(v_1, v_2, \ldots, v_{d-1})|\lesssim 1/r, \quad v_d\lesssim 1/r.\]
	Hence we can precisely set $y=(v_1,v_2\ldots,v_{d-1})$ and assume
	\begin{equation} \label{E:Quasi-orthogonality-8}
		|y|\leq c_5/r, \quad |v_d| \leq c_6/r
	\end{equation}
	with the constants $c_5, c_6$ to be determined later. Under these assumptions, in the direction $e_{d+1}$, we can use the multi-variable Taylor's theorem to estimate the largest displacement between the tangent plane and the surface $(\xi,|\xi|^{\alpha}/2^{\alpha-1})$ as follows
	\[|(y, k+v_d)|^{\alpha}/2^{\alpha}- \l(k^{\alpha}/2^{\alpha-1}+ \alpha v_d k^{\alpha-1}/2^{\alpha-1}\r) \lesssim |(y, v_d)|^2\leq (c_5^2+c_6^2)/r^2.\]
	Here we have used the fact $k\sim 1$ which leads to all the second-order partial derivatives and the Hessian of this corresponding function $g(\xi)=|\xi|^{\alpha}/2^{\alpha-1}$ comparable to $1$. Hence we can choose some universal constants $c_5, c_6$ small enough such that this displacement is less than $\frac{c_3}{2r^2}$. This conclusion further gives the desired universal constant $K_1$ comparable to $(c_5)^{1-d}(c_6)^{-1}$, such that the rectangle $R_{\tau,\tau'}$ can be decomposed into $K_1$ smaller rectangles
	\[R_{\tau,\tau'}= \bigcup_{\ell=1}^{K_1} R_{\ell},\]
	where the smaller rectangles $R_{\ell}=R_{\ell}(\tau,\tau')$ have the same size with disjoint interiors and satisfy the condition \eqref{E:Quasi-orthogonality-8}.
	
	Then for each $\ell$, let $c_{\ell}$ be the center of the rectangle $P_{\ell}$ and let $T(R_{\ell})$ denote the lift of $R_{\ell}$ into the tangent plane $T(c_{\ell})$. Define the parallelepiped $P_{\ell}=P_{\ell}(\tau,\tau') \subset \bR^{d+1}$ as the sum-set
	\[P_{\ell}:=T(R_{\ell})+\l\{s e_{d+1}: \frac{c_2}{r^2}\leq s\leq \frac{3c_3}{2r^2}\r\}.\]
	Note that distinct $P_{\ell}$ have disjoint interiors and the following inclusion holds
	\[\l(\tilde{\tau}+\tilde{\tau}'\r)\subset \Gamma_{\tau, \tau'} \subset \l(\bigcup_{\ell=1}^{K_1} P_{\ell}\r).\]
	Moreover, by the construction of $R_{\tau,\tau'}^{*}$ and $R_{\ell}$, there exists $\beta>0$ such that $(1+\beta)\cdot R_{\ell} \subset R_{\tau,\tau'}^{*}$ holds for every $\ell\in \{1,2,\ldots,K_1\}$. Taking the aforementioned displacements into consideration, recalling the choice of constants $c_5$ and $c_6$, we can guarantee that
	\[(1+\beta)\cdot P_{\ell} \subset \Sigma_{\tau, \tau'},\]
	possibly choosing a smaller $\beta$ but still universal. This finishes the proof of our desired claim.
\end{proof}

One basic refinement of the original $\alpha$-Strichartz inequality \eqref{E:alpha-Strichartz} is the following annular orthogonality Lemma \ref{L:Annular orthogonality}. Indeed, a direct application of the Littlewood-Paley theory and the mixed-norm estimates \eqref{E:alpha-Strichartz} will give this refinement. Similar arguments can be found in \cite[Lemma 2.3]{JSS2017}. See also \cite[Proposition 2.1]{COSS2021} and \cite[Lemma 4.1]{KSV2012}. We omit the proof here for simplicity.
\begin{lemma}\label{L:Annular orthogonality}
	Let $d\geq 2$. We have the following annular orthogonality estimate
	\[\l\|[E_{\alpha}]f\r\|_{L_{t,x}^{q_0}(\bR^{d+1})}^{q_0} \lesssim \sup_{N \in 2^{\bZ}} \l\|[E_{\alpha}]f_N\r\|_{L_{t,x}^{q_0}(\bR^{d+1})}^{q_0-2} \|f\|_{L_x^2(\bR^d)}^2\]
	for every $f\in L^2(\bR^d)$, where $\widehat{f}_N:=\mathds{1}_{\mathcal{A}_N}\widehat{f}$.
\end{lemma}

Based on this annular orthogonality result, we are able to focus on the annular case and aim to establish some suitable control for the item $\l\|[E_{\alpha}]f_N \r\|_{L_{t,x}^{q_0}}$. At this point, the advantage is that the truncated surface has non-zero Gaussian curvature. Hence we are able to apply Tao's classical bilinear restriction estimates in \cite{Tao2003}.
\begin{thm} \label{T:Bilinear estimates}
	Suppose that $\frac{d+3}{d+1}<p<\frac{d+2}{d}$ and $j_0\in \{1,2\ldots, K_{d,\alpha}\}$. Then for every $\tau\sim \tau'\subset \mathcal{A}_1^{j_0}$, the following bilinear estimates holds
	\[\l\|[e^{it|\nabla|^{\alpha}}]f_{\tau}\cdot [e^{it|\nabla|^{\alpha}}]f_{\tau'}\r\|_{L_{t,x}^p} \lesssim |\tau|^{1-\frac{q_0}{2p}} \|f_{\tau}\|_{L^2}\|f_{\tau'}\|_{L^2}\]
	for all $f\in L^2(\bR^d)$ satisfying $\supp{\widehat{f}}\subset \mathcal{A}_1$. Therefore, by interpolation there exists $s_0\in (1,2)$ such that
	\[\l\|[e^{it|\nabla|^{\alpha}}]f_{\tau}\cdot [e^{it|\nabla|^{\alpha}}]f_{\tau'}\r\|_{L_{t,x}^{q_0/2}} \lesssim |\tau|^{1-\frac{2}{s_0}} \|\widehat{f}_{\tau}\|_{L^{s_0}}\|\widehat{f}_{\tau'}\|_{L^{s_0}}.\]
\end{thm}
\begin{proof}[\textbf{Proof of Theorem \ref{T:Bilinear estimates}}]
	This result follows from a standard parabolic rescaling argument by Tao's bilinear estimates on the paraboloid \cite{Tao2003}. Note that our surface $(\xi, |\xi|^{\alpha})$ restricted on the unit cube-annular $\mathcal{A}_1$ is a elliptic-type compact surface in the sense of \cite[Section 2]{TVV1998}. Hence we can use the bilinear estimates \cite[Section 9, third remark]{Tao2003} to deduce this desired conclusion. See \cite[Theorem A.1]{FLS2016} for further details on this parabolic rescaling argument.
\end{proof}

As mentioned above, on each annular, we are going to apply this bilinear estimate Theorem \ref{T:Bilinear estimates} to deduce the following annular refined $\alpha$-Strichartz estimate Lemma \ref{L:Annular refined alpha-Strichartz}. The arguments (from bilinear-type estimate to refined-type Strichartz estimate) are calssical, and similar arguments can be seen in the works such as \cite{BV2007,FS2018,Tao2009}. Here we follow the ideas in these works, and show the details for the convenience of the readers.
\begin{lemma}\label{L:Annular refined alpha-Strichartz}
	Let $d\geq 2$. There exists $\gamma\in\l(0,1\r)$ such that the following estimate
	\[\l\|[e^{it|\nabla|^{\alpha}}]f\r\|_{L_{t,x}^{q_0}} \lesssim \l[\sup_{r\in 2^{\bZ_{+}}} \sup_{\tau\in\mathcal{D}_{1,r}} |\tau|^{-1/2} \l\|[e^{it|\nabla|^{\alpha}}] f_{\tau}\r\|_{L_{t,x}^{\infty}}\r]^{\gamma} \|f\|_{L_x^2}^{1-\gamma}\]
	holds for all $f\in L^2(\bR^d)$ satisfying $\supp{\widehat{f}} \subset \mathcal{A}_1$. Moreover, by rescaling, the following estimate
	\[N^{\frac{\alpha-2}{q_0}} \l\|[e^{it|\nabla|^{\alpha}}]f\r\|_{L_{t,x}^{q_0}} \lesssim \l[\sup_{r\in 2^{\bZ_{+}}} \sup_{\tau\in\mathcal{D}_{N,r}} |\tau|^{-1/2} \l\|[e^{it|\nabla|^{\alpha}}] f_{\tau}\r\|_{L_{t,x}^{\infty}}\r]^{\gamma} \|f\|_{L_x^2}^{1-\gamma}\]
	holds for all $f\in L^2(\bR^d)$ satisfying $\supp{\widehat{f}} \subset \mathcal{A}_N$.
\end{lemma}

\begin{proof}[\textbf{Proof of Lemma \ref{L:Annular refined alpha-Strichartz}}]
	Due to the relation \eqref{E:Angular divided}, it suffices to fix some $j_0\in\{1,2,\ldots,K_{d,\alpha}\}$ and prove that
	\begin{equation}\label{E:Annular refined-1}
		\l\|[e^{it|\nabla|^{\alpha}}] f_{\mathcal{A}_1^{j_0}}\r\|_{L_{t,x}^{q_0}} \lesssim \l[\sup_{r\in 2^{\bZ_{+}}} \sup_{\tau\in\mathcal{D}_{1,r}} |\tau|^{-1/2} \l\|[e^{it|\nabla|^{\alpha}}] f_{\tau}\r\|_{L_{t,x}^{\infty}}\r]^{\gamma} \|f\|_{L_x^2}^{1-\gamma}.
	\end{equation}
	Our strategy is using the bilinear estimates Theorem \ref{T:Bilinear estimates}. Thus, the Whitney-type decomposition gives
	\[\l\|[e^{it|\nabla|^{\alpha}}] f_{\mathcal{A}_1^{j_0}}\r\|_{L_{t,x}^{q_0}} \leq \l\|\sum_{\tau\sim\tau' \subset \mathcal{A}_1^{j_0}} [e^{it|\nabla|^{\alpha}}] f_{\tau} \cdot [e^{it|\nabla|^{\alpha}}] f_{\tau'}\r\|_{L_{t,x}^{q_0/2}}^{1/2}.\]
	Then the quasi-orthogonality Lemma \ref{L:Quasi-orthogonality} and H\"older's inequality imply that
	\begin{align}
		\l\|\sum_{\tau\sim\tau' \subset \mathcal{A}_1^{j_0}} [e^{it|\nabla|^{\alpha}}] f_{\tau} \cdot [e^{it|\nabla|^{\alpha}}] f_{\tau'}\r\|_{L_{t,x}^{q_0/2}}^{1/2} &\lesssim \l( \sum_{\tau\sim\tau' \subset \mathcal{A}_1^{j_0}} \l\|[e^{it|\nabla|^{\alpha}}] f_{\tau} \cdot [e^{it|\nabla|^{\alpha}}] f_{\tau'} \r\|_{L_{t,x}^{q_0/2}}^{q_0} \r)^{\frac{1}{2q_0}} \notag\\
		&\leq \sup_{\tau\sim \tau' \subset \mathcal{A}_1^{j_0}} \l\|[e^{it|\nabla|^{\alpha}}] f_{\tau} \cdot [e^{it|\nabla|^{\alpha}}] f_{\tau'} \r\|_{L_{t,x}^{q_0/2}}^{\frac{q_0-s}{2q_0}} \label{E:Annular refined-2}\\
		&\quad \times \l(\sum_{\tau\sim\tau' \subset \mathcal{A}_1^{j_0}} \l\|[e^{it|\nabla|^{\alpha}}] f_{\tau} \cdot [e^{it|\nabla|^{\alpha}}] f_{\tau'} \r\|_{L_{t,x}^{q_0/2}}^{s}\r)^{\frac{1}{2q_0}}.\label{E:Annular refined-3}
	\end{align}
	Here the inequality is valid for all $0<s<q_0$, though we will choose $1<s<q_0$ later. In the remainder of this proof, we will estimate \eqref{E:Annular refined-2} and \eqref{E:Annular refined-3} by different ways which will finally lead to the desired result \eqref{E:Annular refined-1}.
	
	Firstly for the term \eqref{E:Annular refined-2}, we have that
	\begin{align*}
		\l\|[e^{it|\nabla|^{\alpha}}] f_{\tau} \cdot [e^{it|\nabla|^{\alpha}}] f_{\tau'} \r\|_{L_{t,x}^{q_0/2}} &\leq \l(|\tau|^{-1} \l\|[e^{it|\nabla|^{\alpha}}] f_{\tau} \cdot [e^{it|\nabla|^{\alpha}}] f_{\tau'} \r\|_{L_{t,x}^{\infty}}\r)^{1-\frac{2p}{q_0}} \\
		&\quad \times \l(|\tau|^{\frac{q_0}{2p}-1} \l\|[e^{it|\nabla|^{\alpha}}] f_{\tau} \cdot [e^{it|\nabla|^{\alpha}}] f_{\tau'} \r\|_{L_{t,x}^p}\r)^{\frac{2p}{q_0}},
	\end{align*}
	where $\frac{d+3}{d+1}<p<\frac{d+2}{d}$. Then the bilinear estimates Theorem \ref{T:Bilinear estimates} gives the following estimate
	\[\l\|[e^{it|\nabla|^{\alpha}}] f_{\tau} \cdot [e^{it|\nabla|^{\alpha}}] f_{\tau'} \r\|_{L_{t,x}^{q_0/2}} \lesssim \l(|\tau|^{-1/2} \|[e^{it|\nabla|^{\alpha}}] f_{\tau}\|_{L_{t,x}^\infty}\r)^{1-\frac{2p}{q_0}} \l(|\tau'|^{-1/2} \|[e^{it|\nabla|^{\alpha}}]f_{\tau'}\|_{L_{t,x}^\infty}\r)^{1-\frac{2p}{q_0}} \|f\|_{L_x^2}^{\frac{4p}{q_0}},\]
	which immediately means
	\begin{equation}\label{E:Annular refined-4}
		\sup_{\tau\sim \tau' \subset \mathcal{A}_1^{j_0}} \l\|[e^{it|\nabla|^{\alpha}}] f_{\tau} \cdot [e^{it|\nabla|^{\alpha}}] f_{\tau'} \r\|_{L_{t,x}^{q_0/2}} \lesssim \l[\sup_{r\in 2^{\bZ_{+}}} \sup_{\tau\in\mathcal{D}_{1,r}} |\tau|^{-1/2} \|[e^{it|\nabla|^{\alpha}}] f_{\tau}\|_{L_{t,x}^{\infty}}\r]^{2-\frac{4p}{q_0}} \|f\|_{L_x^2}^{\frac{4p}{q_0}}.
	\end{equation}
	For the term \eqref{E:Annular refined-3}, a standard application of bilinear estimates Theorem \ref{T:Bilinear estimates} concludes that
	\begin{align*}
		\sum_{\tau\sim\tau' \subset \mathcal{A}_1^{j_0}} \l\|[e^{it|\nabla|^{\alpha}}] f_{\tau} \cdot [e^{it|\nabla|^{\alpha}}] f_{\tau'} \r\|_{L_{t,x}^{q_0/2}}^{s} &\lesssim \sum_{\tau\sim\tau' \subset \mathcal{A}_1^{j_0}} \l[|\tau|^{1-2/s_0} \|\widehat{f}_{\tau}\|_{L^{s_0}} \|\widehat{f}_{\tau'}\|_{L^{s_0}}\r]^s \\
		& \lesssim \sum_{\tau\sim\tau' \subset \mathcal{A}_1^{j_0}} \l[|\tau|^{1-2/s_0} \|\widehat{f}_{\tau}\|_{L^{s_0}}^2\r]^s \\
		& \lesssim \sum_{\tau \subset \mathcal{A}_1} \l[|\tau|^{1-2/s_0} \|\widehat{f}_{\tau}\|_{L^{s_0}}^2\r]^s \\
		&\lesssim \|f\|_{L^2}^{2s},
	\end{align*}
	where in the last inequality we have used \cite[Theorem 1.3]{BV2007} with $s>1$, see also \cite[p. 279]{Tao2009}. Thus, this estimate and the previous \eqref{E:Annular refined-4} directly imply our desired conclusion \eqref{E:Annular refined-1}.
\end{proof}

Finally, we are able to show our main result in this section, which is the following refined $\alpha$-Strichartz proposition.
\begin{prop}\label{P:Refined alpha-Strichartz}
	Let $d\geq2$. There exists $\theta\in (0,1)$ such that the following estimate
	\begin{equation}\label{P:Refined alpha-Strichartz-1}
		\l\|[E_{\alpha}]f\r\|_{L_{t,x}^{q_0}} \lesssim \l[\sup_{Q\in\mathcal{D}} |Q|^{-\frac{1}{2}} \l\|[e^{it|\nabla|^{\alpha}}]f_Q\r\|_{L_{t,x}^{\infty}}\r]^{\theta} \|f\|_{L_x^2}^{1-\theta}
	\end{equation}
	holds for all $f\in L^2(\bR^d)$.
\end{prop}
\begin{proof}[\textbf{Proof of Proposition \ref{P:Refined alpha-Strichartz}}]
	Combining the annular orthogonality Lemma \ref{L:Annular orthogonality} and annular refined $\alpha$-Strichartz Lemma \ref{L:Annular refined alpha-Strichartz}, we directly obtain this conclusion by applying Bernstein's inequality.
\end{proof}

There are standard arguments to deduce the following non-zero weak limit result Corollary \ref{C:Non-zero weak limit} from the aforementioned Proposition \ref{P:Refined alpha-Strichartz}. In fact, the $L_{t,x}^{\infty}$-norm in \eqref{P:Refined alpha-Strichartz-1} gives space-time translation parameters and the supremum of dyadic cubes gives scaling-frequency parameters. Readers can see \cite[Corollary 3.2]{FLS2016} for further details and the detailed proof are omitted here for simplicity.
\begin{coro}\label{C:Non-zero weak limit}
	Let $(f_n)$ be a bounded sequence in $L^2(\bR^d)$ with $d\geq 2$. If $\|[E_{\alpha}]f_n\|_{L_{t,x}^{q_0}}\nrightarrow 0$, then there exists $(t_n,x_n,\xi_n, h_n) \subset \bR\times \bR^d \times \bR^d \times \bR_{+}$ with $h_n\|\xi_n\|_{\min}\geq 1/2$, such that up to subsequences
	\[\widehat{[g_n]f_n}\l(\xi+h_n\xi_n\r) \rightharpoonup \widehat{V}\]
	in weak topology of $L^2(\bR^d)$ with $\widehat{V}\neq 0$. Moreover, if $\|f_n\|_{L_x^2}\leq B$ and $\limsup_{n\to\infty} \|E_{\alpha}f_n\|_{L_{t,x}^{q_0}}\geq A$, then we have
	\[\|V\|_{L_x^2}\geq CA^{\beta}B^{-\gamma},\]
	where $C, \beta, \gamma$ depend only on the dimension $d$ and $\alpha$.
\end{coro}

\section{Local smoothing and local convergence} \label{S:Local smoothing and local convergence}
In this section we show some local smoothing property for the solution $[e^{it|\nabla|^{\alpha}}]f$. Then this property further deduces some local convergence and pointwise convergence results which will provide the desired input for Br\'ezis-Lieb type lemma \cite[Lemma 3.1]{FLS2016}. Similar arguments can also be found in \cite{FLS2016,FS2018}. Our local smoothing consequence is the following lemma.
\begin{lemma}\label{L:Local smooth}
	Let $\phi \in L^1(\bR^d)$ be a Schwartz function. Then for all $f\in L^2(\bR^d)$, there holds
	\[\int_{\bR^{d+1}} \phi(x) \l|[D^{\frac{\alpha-1}{2}}] [e^{it|\nabla|^{\alpha}}]f(x)\r|^2 \ddd x\ddd t \lesssim_{\phi} \|f\|_{L^2}.\]
\end{lemma}
\begin{proof}[\textbf{Proof of Lemma \ref{L:Local smooth}}]
	We will mimic the proofs in \cite[Lemma 4.4]{FLS2016} and \cite[Lemma 4.3]{FS2018}, albeit with a small twist.
	
	Let $\mathscr{F}_t$ be the Fourier transform on time space. Then there holds
	\[\mathscr{F}_te^{ia_0t}(\lambda)=2\pi \delta(\lambda-a_0).\]
	Here $\delta$ denotes the Dirac delta function. Thus, Plancherel theorem implies the following estimate
	\[\int_{\bR^{d+1}} \phi(x) \l|[D^{\frac{\alpha-2}{2}}] [e^{it|\xi|^{\alpha}}]f(x)\r|^2 \ddd x\ddd t \sim \int_{\bR^d} \int_{\bR^d} \widehat{\phi}(\xi'-\xi) |\xi|^{\frac{\alpha-1}{2}} |\xi'|^{\frac{\alpha-1}{2}} \widehat{f}(\xi)\overline{\widehat{f}(\xi')} \delta(|\xi|^{\alpha}-|\xi'|^{\alpha}) \ddd\xi \ddd \xi'.\]
	By Schur test, we only need to show
	\begin{equation*}
		\sup_{\xi} \int_{\bR^d} \widehat{\phi}(\xi'-\xi) |\xi|^{\frac{\alpha-1}{2}} |\xi'|^{\frac{\alpha-1}{2}} \delta(|\xi|^{\alpha}-|\xi'|^{\alpha}) \ddd\xi' <\infty.
	\end{equation*}
	Setting polar coordinate $\xi'=k\theta$ with $\theta\in \bS^{d-1}$ and denoting $\xi= |\xi|\theta_{\xi}$, it remains to deduce
	\begin{equation*}
		\sup_{\xi} |\xi|^{d-1} \int_{\bS^{d-1}} \l|\widehat{\phi} \Big(|\xi|(\theta_{\xi}-\theta)\Big)\r| \ddd \theta<\infty.
	\end{equation*}
	The dominated convergence theorem implies that $\int_{\bS^{d-1}} \l|\widehat{\phi} \Big(|\xi|(\theta_{\xi}-\theta)\Big)\r| \ddd \theta$ is a continuous function with respect to $\xi$, hence it suffices to prove
	\begin{equation}\label{E:Local smooth-1}
		\int_{\bS^{d-1}} \l|\widehat{\phi} \Big(|\xi|(\theta_{\xi}-\theta)\Big)\r| \ddd \theta=O\l(|\xi|^{1-d}\r)
	\end{equation}
	as $|\xi|\to \infty$. A changing of variables gives
	\[\int_{\bS^{d-1}} \l|\widehat{\phi} \Big(|\xi|(\theta_{\xi}-\theta)\Big)\r| \ddd \theta= \frac{1}{|\xi|^{d-1}} \int_{\bS^{d-1}_{\xi}} |\widehat{\phi}(\theta)| \ddd\theta,\]
	where $\bS^{d-1}_{\xi}:=\l\{\theta\in\bR^d: |\theta-\theta_{\xi}|=|\xi|\r\}$ is the sphere centered at $\theta_{\xi}$ with radius $|\xi|$. Then the fact that $\phi$ is a Schwartz function deduces the desired result \eqref{E:Local smooth-1} and finishes the proof.
\end{proof}

\begin{lemma}\label{L:Local convergence}
	Let $\beta<\frac{\alpha-1}{2}$ and define the operator $[\bar{E}_{\beta}]:=[D^{\beta}][e^{it|\nabla|^{\alpha}}]$. For a bounded sequence of functions $(f_n)$ in $L^2(\bR^d)$, if we have
	\[f_n\rightharpoonup 0\]
	weakly in $L^2(\bR^d)$, then up to subsequences there holds $[\bar{E}_{\beta}]f_n(t,x)\to 0$ almost everywhere in $\bR^{d+1}$.
\end{lemma}
\begin{proof}[\textbf{Proof of Lemma \ref{L:Local convergence}}]
	Based on the local smoothing Lemma \ref{L:Local smooth}, the proof is standard. Similar proofs can also be found in \cite[Proposition 4.3]{FLS2016}, \cite[Lemma 4.3]{FS2018} and \cite[Lemma D.1]{FS2018}.
	
	We aim to show that $[\bar{E}_{\beta}] f_n\to 0$ strongly in $L^2_{loc}(\bR^{d+1})$, which implies the desired result by a Cantor diagonal argument. Let $K\subset \bR^{d+1}$ be a compact set and $\varepsilon>0$. For $\Lambda>0$, define
	\[[P_{\Lambda}] f(x):=\mathscr{F}^{-1} \mathds{1}_{B_{\Lambda}} \mathscr{F}f(x), \quad [P_{\Lambda}^{\bot}]f(x):=f(x)-[P_{\Lambda}] f(x).\]
	Here $B_{\Lambda}$ is the ball in $\bR^d$ centered at origin with radius $\Lambda$. For the $[P_{\Lambda}^{\bot}]$ term, since $\beta<\frac{\alpha-1}{2}$, we use Lemma \ref{L:Local smooth} to conclude that
	\begin{align*}
		\l\|\mathds{1}_{K} [P_{\Lambda}^{\bot}] [\bar{E}_{\beta}] f_n\r\|_{L_{t,x}^2} &\leq \l\|\mathds{1}_{K} e^{|x|^2}\r\|_{L_{t,x}^{\infty}} \l\|e^{-|x|^2} [D^{\frac{\alpha-1}{2}}] [e^{it|\nabla|^{\alpha}}]\r\|_{L_x^2\to L_{t,x}^2} \\
		&\quad \times\l\|[P_{\Lambda}^{\bot}] [D^{\beta-\frac{\alpha-1}{2}}]\r\|_{L_x^2\to L_x^2} \|f_n\|_{L_x^2} \\
		&\leq C_K \Lambda^{\beta-\frac{\alpha-1}{2}},
	\end{align*}
	where $C_K$ is independent of $n$ and the notation $\|[T]\|_{L^p\to L^q}$ denotes the norm of the operator $[T]$. Hence we can choose $\Lambda$ large enough independent of $n$ such that
	\[\l\|\mathds{1}_{K} [P_{\Lambda}^{\bot}] [\bar{E}_{\beta}] f_n\r\|_{L_{t,x}^2} \leq \varepsilon.\]
	Then for this large $\Lambda$ and the $[P_{\Lambda}]$ term, we claim that for every fixed $t\in\bR$ there holds
	\begin{equation}\label{E:Local convergence-1}
		\mathds{1}_{K}(t,\cdot) [P_{\Lambda}] [D^{\beta}] [e^{it|\nabla|^{\alpha}}] f_n \to 0
	\end{equation}
	strongly in $L_x^2(\bR^d)$ as $n\to\infty$. Indeed, for any $(t,x)\in \bR\times\bR^d$, since $f_n\rightharpoonup 0$ in $L_x^2(\bR^d)$ we have
	\[\mathds{1}_{K}(t,\cdot) [P_{\Lambda}] [D^{\beta}] [e^{it|\nabla|^{\alpha}}] f_n (x)= \mathds{1}_{K}(t,x) \int_{\bR^d} \mathds{1}_{\{|\xi|\leq \Lambda\}} |\xi|^{\beta} e^{ix\xi+it|\xi|^{\alpha}} \widehat{f}_n(\xi) \ddd\xi \to 0\]
	as $n\to \infty$. Meanwhile, by Cauchy-Schwarz and the boundedness of $\|f_n\|_{L_x^2}$, there holds
	\[\Big|\mathds{1}_{K}(t,\cdot) [P_{\Lambda}] [D^{\beta}] [e^{it|\nabla|^{\alpha}}] f_n(x)\Big| \lesssim \mathds{1}_K(t,x) \Lambda^{\beta+d/2} \|f_n\|_{L_x^2}\lesssim_{\Lambda} \mathds{1}_{K}(t,x).\]
	Therefore, dominated convergence theorem gives the desired claim \eqref{E:Local convergence-1}. Notice that there holds
	\[\l\|\mathds{1}_{K}(t,\cdot) [P_{\Lambda}] [D^{\beta}] [e^{it|\nabla|^{\alpha}}] f_n(x)\r\|_{L_{t,x}^2} \lesssim_{\Lambda} \|\mathds{1}_{K}\|_{L_{t,x}^2}.\]
	Hence, the claim \eqref{E:Local convergence-1} and dominated convergence theorem imply that $\mathds{1}_{K} [P_{\Lambda}] [\bar{E}_{\beta}] f_n \to 0$ strongly in $L_{t,x}^2(\bR^{d+1})$ as $n\to\infty$. This completes the proof.
\end{proof}

\section{Approximate symmetry}\label{S:Approximate symmetry}
As mentioned in the introduction, to establish the main precompactness result Theorem \ref{T:Precompactness}, one should understand the behavior of \textit{approximate symmetries} in the sense of \cite[Remark 2.6]{FS2018}. In other words, we should investigate the effect of frequency parameters $\xi_n$. In this section, for frequencies $(\xi_n)\subset\bR^d$ with $|\xi_n|\to \infty$, then up to subsequences we introduce the following notations
\[\bar{\xi}_n:=\frac{\xi_n}{|\xi_n|}, \quad \xi_0:=\lim_{n\to\infty} \bar{\xi}_n.\]
For an unit vector $\xi_0\in \bS^{d-1}$, we define the linear transformation $A_0$ on $\bR^d$ as follows
\[A_0: \xi\mapsto \xi^{\bot}\sqrt{\alpha/2}+ \xi^{\shortparallel}\sqrt{\alpha(\alpha-1)/2}, \quad \xi^{\shortparallel}:=(\xi\xi_0) \xi_0, \quad \xi^{\bot}:=\xi-\xi^{\shortparallel};\]
and the associated unitary operator $[\tilde{A}_0]$ on $L^2(\bR^d)$ is defined by
\begin{equation}\label{E:Approximate unitary operator}
	[\tilde{A}_0]f(x):=|\det A_0|^{1/2} f(A_0x).
\end{equation}
Note that the absolute value of determinate $|\det A_0|=(\alpha/2)^{d/2} \sqrt{\alpha-1}$. The approximate symmetry mainly investigates the operator
\begin{equation}\label{E:Approximate operator}
	[T_{\alpha}^n]f(t,x):=\int_{\bR^d} |\xi+\xi_n|^{\frac{\alpha-2}{q_0}} e^{ix\xi +it|\xi+\xi_n|^{\alpha}} \widehat{f}(\xi) \ddd \xi
\end{equation}
and the following function
\begin{equation}\label{E:Approximate function}
	F_{\alpha}^n(t,x)=[\bar{T}_{\alpha}^n]f(t,x):=\int_{\bR^d} \l|\frac{\xi}{|\xi_n|}+\bar{\xi}_n\r|^{\frac{\alpha-2}{q_0}} e^{ix\xi+it\Phi_n(\xi)} \widehat{f}(\xi) \ddd\xi,
\end{equation}
where
\[\Phi_n(\xi):=\frac{1}{|\xi_n|^{\alpha-2}} \l[|\xi+\xi_n|^{\alpha}-|\xi_n|^{\alpha}- \alpha|\xi_n|^{\alpha-2}\xi_n\xi\r].\]
Indeed, it is not hard to see that there holds
\[\|[T_{\alpha}^n]f\|_{L_{t,x}^{q_0}}=\l\|F_{\alpha}^n\r\|_{L_{t,x}^{q_0}}.\]
Our first result on the aforementioned approximate operator $[T_{\alpha}^n]$ is the following asymptotic Schr\"odinger behavior lemma.

\begin{lemma}\label{L:Asymptotic schrodinger}
	Let $f\in L^2(\bR^d)$ with $\|f\|_{L^2}=1$ and frequencies $(\xi_n)\subset\bR^d$. Set
	\[\widehat{f}_{n}(\xi):=\widehat{f}(\xi-\xi_n), \quad \mathbf{a}_{d,\alpha}^{*}:=(\alpha-1)^{\frac{-1}{2d+4}} (\alpha/2)^{\frac{-d}{2d+4}}.\]
	Then, up to subsequences, we have the following asymptotic Schr\"odinger behavior
	\[\lim_{|\xi_n|\to\infty} \l\|[E_{\alpha}] f_{n}\r\|_{L_{t,x}^{q_0}}=\lim_{n\to\infty} \l\|[T_{\alpha}^n] f\r\|_{L_{t,x}^{q_0}}=\mathbf{a}_{d,\alpha}^{*} \l\|[e^{it\Delta}][\tilde{A}_0]f\r\|_{L_{t,x}^{q_0}}.\]
\end{lemma}
\begin{remark}\label{R:Asymptotic schrodinger}
	Based on the existence of extremals for $\mathbf{S}_d^{*}$ by \cite{Shao2009EJDE}, this Lemma \ref{L:Asymptotic schrodinger} immediately gives
	\[\mathbf{M}_{d,\alpha}\geq (\alpha-1)^{\frac{-1}{2d+4}} (\alpha/2)^{\frac{-d}{2d+4}} \mathbf{S}_d^{*}.\]
	For the case $d=1$, after some accurate numerical calculations, it has been shown in \cite[Theorem 1.4]{BOQ2020} that there exists $\alpha_1\approx 5.485$ such that the strict inequality \eqref{T:Precompactness-1} holds for indices $\alpha\in (1,\alpha_1)\setminus \{2\}$. For the case $d=2$, it has been proved in \cite[Proposition 6.9]{OQ2018} that there exists $\alpha_0>5$ such that the strict inequality \eqref{T:Precompactness-1} holds for $\alpha\in(2,\alpha_0)$. These consequences suggest that the aforementioned strict inequality might hold for all general $(d,\alpha)$ with $\alpha>2$. However, it remains an open question for us.
\end{remark}

\begin{proof}[\textbf{Proof of Lemma \ref{L:Asymptotic schrodinger}}]
	The change of variables can yield that
	\begin{equation}\label{E:Asmptotic schrodinger-0.5}
		\l\|[E_{\alpha}] f_{n}\r\|_{L_{t,x}^{q_0}}= \l\|[T_{\alpha}^n] f\r\|_{L_{t,x}^{q_0}}, \quad \|[\bar{T}_{\alpha}^n]\|_{L_x^2\to L_{t,x}^{q_0}}=\|[T_{\alpha}^n]\|_{L_x^2\to L_{t,x}^{q_0}} =\l\|[E_{\alpha}]\r\|_{L_x^2\to L_{t,x}^{q_0}}= \mathbf{M}_{d,\alpha}.
	\end{equation}
	Hence by some dense argument we may assume $f$ has compact Fourier support at the beginning. Using our aforementioned definition \eqref{E:Approximate function}, we can obtain that
	\begin{align*}
		\|[T_{\alpha}^n]f\|_{L_{t,x}^{q_0}} &=\l\|\int_{\bR^d} |\xi_n|^{\frac{\alpha-2}{q_0}} \l|\frac{\xi}{|\xi_n|}+\bar{\xi}_n\r|^{\frac{\alpha-2}{q_0}} e^{ix\xi+it|\xi+\xi_n|^{\alpha}} \widehat{f}(\xi) \ddd\xi\r\|_{L_{t,x}^{q_0}} \\
		&= \l\|\int_{\bR^d} \l|\frac{\xi}{|\xi_n|}+\bar{\xi}_n\r|^{\frac{\alpha-2}{q_0}} e^{ix\xi+it\frac{|\xi+\xi_n|^{\alpha}}{|\xi_n|^{\alpha-2}}} \widehat{f}(\xi) \ddd\xi\r\|_{L_{t,x}^{q_0}} \\&= \l\|F_{\alpha}^n\r\|_{L_{t,x}^{q_0}}.
	\end{align*}
	Then the Taylor's theorem, paired with the assumptions $|\xi_n|\to \infty$ and $f$ has compact Fourier support, will directly imply the following pointwise convergence
	\begin{equation}\label{E:Asmptotic schrodinger-0.8}
		\lim_{n\to\infty} \Phi_n(\xi)=\frac{\alpha |\xi|^2 +\alpha(\alpha-2)|\xi\xi_0|^2}{2}.
	\end{equation}
	Therefore it is reasonable to expect the following estimate
	\begin{equation} \label{E:Asymptotic schrodinger-1}
		\lim_{n\to\infty} \Big\|[T_{\alpha}^n]f\Big\|_{L_{t,x}^{q_0}} = \l\|\int_{\bR^d} e^{ix\xi+it\l[\alpha |\xi|^2/2 +\alpha(\alpha-2)|\xi\xi_0|^2/2\r]} \widehat{f}(\xi) \ddd \xi\r\|_{L_{t,x}^{q_0}}.
	\end{equation}
	Let us postpone the proof of this result and go ahead by using this estimate \eqref{E:Asymptotic schrodinger-1}. From the definition of $A_0$ and $[\tilde{A}_0]$, we know that $|A_0\xi|^2=\alpha |\xi|^2/2 +\alpha(\alpha-2)|\xi\xi_0|^2/2$. Then a direct computation yields that
	\begin{align*}
		\l\|\int_{\bR^d} e^{ix\xi+it\l[\alpha |\xi|^2/2 +\alpha(\alpha-2)|\xi\xi_0|^2/2\r]} \widehat{f}(\xi) \ddd \xi\r\|_{L_{t,x}^{q_0}} &= \l\|\int_{\bR^d} e^{ix\xi+it|A_0\xi|^2} \widehat{f}(\xi) \ddd \xi\r\|_{L_{t,x}^{q_0}} \\
		&= |\det A_0|^{-1/2} \l\|[e^{it\Delta}][\tilde{A}_0]f (A_0^{-1}x)\r\|_{L_{t,x}^{q_0}} \\
		&=|\det A_0|^{1/{q_0}-1/2} \l\|[e^{it\Delta}][\tilde{A}_0]f\r\|_{L_{t,x}^{q_0}}.
	\end{align*}
	Finally, the fact $|\det A_0|=(\alpha/2)^{d/2} \sqrt{\alpha-1}$ gives the constant $\mathbf{a}_{d,\alpha}^{*}=(\alpha-1)^{\frac{-1}{2d+4}} (\alpha/2)^{\frac{-d}{2d+4}}$.
	
	Now it remains to prove \eqref{E:Asymptotic schrodinger-1}. Indeed, we can obtain the following stronger result
	\begin{equation} \label{E:Asymptotic schrodinger-2}
		F_{\alpha}^n(t,x) \to \int_{\bR^d} e^{ix\xi+it\l[\alpha |\xi|^2/2 +\alpha(\alpha-2)|\xi\xi_0|^2/2\r]} \widehat{f}(\xi) \ddd \xi
	\end{equation}
	strongly in $L_{t,x}^{q_0}(\bR^{d+1})$ as $n\to\infty$. Firstly for any $(t,x)\in\bR^{d+1}$, since $f$ is a Schwartz function with compact Fourier support, dominated convergence theorem and the estimate \eqref{E:Asmptotic schrodinger-0.8} imply
	\begin{equation}\label{E:Asymptotic schrodinger-3}
		\lim_{n\to\infty} F_{\alpha}^n(t,x)= \int_{\bR^d} e^{ix\xi+it\l[\alpha |\xi|^2/2 +\alpha(\alpha-2)|\xi\xi_0|^2/2\r]} \widehat{f}(\xi) \ddd \xi.
	\end{equation}
	On the other hand, for $n$ large enough, by a standard application of stationary phase method we can deduce the following two decay estimates
	\begin{equation}\label{E:Asymptotic schrodinger-4}
		|F_{\alpha}^n(t,x)| \leq C (1+t^2+|x|^2)^{-\frac{d}{4}}
	\end{equation}
	and
	\begin{equation}\label{E:Asymptotic schrodinger-5}
		|H_{\alpha}(t,x)|:=\l|\int_{\bR^d} e^{ix\xi+it\l[\alpha |\xi|^2/2 +\alpha(\alpha-2)|\xi\xi_0|^2/2\r]} \widehat{f}(\xi) \ddd \xi\r| \leq C (1+t^2+|x|^2)^{-\frac{d}{4}},
	\end{equation}
	where the constant $C$ is independent of $n$. Assume for the moment these two decay estimates and let us show the desired convergence result \eqref{E:Asymptotic schrodinger-2} first. From the decay estimates \eqref{E:Asymptotic schrodinger-4} and \eqref{E:Asymptotic schrodinger-5}, after taking a subsequence, we have that for all $n$ and some constant $C'$ there holds
	\[\int_{|(t,x)|>R}\l(|F_n(t,x)|^{q_0}+|H_{\alpha}(t,x)|^{q_0}\r) \ddd t\ddd x \leq \frac{C'}{R}.\]
	By setting $R>0$ large enough, this term can be arbitrary small uniformly in $n$. Hence it suffices to prove that for any fixed $R>0$, there holds
	\[\mathds{1}_{B_R}(t,x)F_{\alpha}^n(t,x) \to \mathds{1}_{B_R}(t,x) H_{\alpha}(t,x)\]
	strongly in $L_{t,x}^{q_0}(\bR^{d+1})$ as $n\to\infty$. By dominated convergence theorem, this follows immediately from the pointwise convergence \eqref{E:Asymptotic schrodinger-3} and the following uniform bound (for $n$ large enough)
	\[|F_{\alpha}^n(t,x)| \lesssim \|\widehat{f}\|_{L^1}<\infty.\]
	Thus, it remains to prove the decay estimates \eqref{E:Asymptotic schrodinger-4} and \eqref{E:Asymptotic schrodinger-5} by means of stationary phase method. In view of the Fourier extension on the surface
	\[\bar{S}:=\l(\xi, \alpha |\xi|^2/2 +\alpha(\alpha-2)|\xi\xi_0|^2/2\r),\]
	then the fact that $\bar{S}$ has non-zero Gaussian curvature and \cite[p. 348, Theorem 1]{Stein1993} directly give the result \eqref{E:Asymptotic schrodinger-5}. Indeed, a direct computation shows that the determinant of corresponding Hessian matrix is a constant $\alpha^d(2\alpha-2)$. To deduce \eqref{E:Asymptotic schrodinger-4}, since $|\xi_n|\to\infty$, we point out that for $n$ large enough there holds
	\[|F_{\alpha}^n(t,x)|\leq 3 \l|\int_{\bR^d} e^{ix\xi+it\Phi_n(\xi)} \widehat{f}(\xi) \ddd\xi\r|.\]
	Let $\bar{\Phi}_n(\bar{x},\bar{t},\xi):=\bar{x}\xi+\bar{t}\Phi_n(\xi)$ with $(\bar{x},\bar{t}):=(x,t)/|(x,t)|$. We investigate the gradient
	\[\nabla \bar{\Phi}_n(\bar{x},\bar{t},\xi)=\bar{x}+\frac{\bar{t}}{|\xi_n|^{\alpha-2}}\l[\alpha|\xi+\xi_n|^{\alpha-2}(\xi+\xi_n)-\alpha|\xi_n|^{\alpha-2}\xi_n\r]\]
	and the Hessian
	\[\mathrm{Hess}\bar{\Phi}_n(\bar{x},\bar{t},\xi)=\frac{\bar{t}}{|\xi_n|^{\alpha-2}}\l(\alpha|\xi+\xi_n|^{\alpha-2}E_d +\alpha(\alpha-2)|\xi+\xi_n|^{\alpha-4} L_n^T L_n\r),\]
	where $E_d$ is the identity matrix, $L_n:=(\xi+\xi_n)$ is viewed as row vector and $L_n^T$ is the transpose of $L_n$. Then we can compute the absolute value for the determinant of this Hessian
	\[\l|\det \mathrm{Hess} \bar{\Phi}_n(\bar{x},\bar{t},\xi)\r|=\frac{(\alpha \bar{t})^d(\alpha-1)|\xi+\xi_n|^{d(\alpha-2)}}{|\xi_n|^{d(\alpha-2)}}.\]
	Since $|\xi_n|\to\infty$ and the $\xi$-localization, we have the following estimates
	\[\lim_{n\to\infty}|\nabla\bar{\Phi}_n|= |\bar{x}+\alpha(\alpha-1)\bar{t}\xi|, \quad \lim_{n\to\infty}\l|\det \mathrm{Hess}\bar{\Phi}_n\r|= (\alpha \bar{t})^d(\alpha-1).\]
	Thus, for $n$ large enough, the proof is divided into two cases. Firstly if $\bar{t}\geq\varepsilon$ for some small $\varepsilon$ to be determined later, then we always have $\l|\det \mathrm{Hess}\bar{\Phi}_n\r| \gtrsim \varepsilon^d$. Hence the uniform stationary phase estimates of Alazard, Burq and Zuily \cite[Theorem 1]{ABZ2017} will give the desired estimate \eqref{E:Asymptotic schrodinger-4}. Secondly if $\bar{t}<\varepsilon$ which means $|\bar{x}|>\sqrt{1-\varepsilon^2}$, duo to the $\xi$-localization, we can choose $\varepsilon$ small enough such that $|\bar{x}|\geq \alpha(\alpha-1)\bar{t}|\xi|+1/8$ uniformly in $\xi$. Then we always have the following uniform estimate
	\[\l|\nabla\bar{\Phi}_n\r|\geq |\bar{x}|-\alpha(\alpha-1)\bar{t}|\xi|\geq 1/8>0.\]
	Therefore, the classical integration by parts arguments \cite[p. 341, Proposition 4]{Stein1993} will deduce an even faster decay than the desired conclusion \eqref{E:Asymptotic schrodinger-4} and complete the proof.
\end{proof}

As we have done in Section \ref{S:Local smoothing and local convergence}, to use the Br\'ezis-Lieb type lemma, analogously the corresponding local smoothing and local convergence results are required to be established. At this point due to the parameters $\xi_n$, we need to do some extra estimates by using the Taylor's theorem.
\begin{lemma}\label{L:Approximate local smooth}
	Suppose that $\phi$ be a Schwartz function and $|\eta|\geq 100$. There exists $C>0$ such that for arbitrary function $f\in L^2(\bR^d)$ which has Fourier support in $\{\xi: |\xi|\leq |\eta|/5\}$, we have
	\[\int_{\bR^{d+1}} \phi(x) \Big|[\psi_{\eta}(D)] [\bar{T}_{\alpha}^{\eta}]f(t,x)\Big|^2 \ddd x\ddd t\leq C\|f\|_{L_x^2}^2,\]
	where $[\bar{T}_{\alpha}^{\eta}]$ is defined as in \eqref{E:Approximate function} with parameter $\eta$ substituting $\xi_n$ and
	\[\psi_{\eta}(\xi):=|\eta|^{\frac{\alpha-2}{q_0}}|\xi|^{1/2}\l|\xi+\eta\r|^{-\frac{\alpha-2}{q_0}}.\]
\end{lemma}
\begin{proof}[\textbf{Proof of Lemma \ref{L:Approximate local smooth}}]
	By Plancherel theorem, as in the proof of Lemma \ref{L:Local smooth}, we investigate
	\[\int_{\bR^d} \int_{\bR^d} \widehat{\phi}(\xi'-\xi) \psi_{\eta}(\xi)\psi_{\eta}(\xi') |\eta|^{\frac{4-2\alpha}{q_0}} \l|\xi+\eta\r|^{\frac{\alpha-2}{q_0}} \l|\xi'+\eta\r|^{\frac{\alpha-2}{q_0}} \overline{\widehat{f}(\xi')} \widehat{f}(\xi) \delta\l[\Phi_{\eta}(\xi)-\Phi_{\eta}(\xi')\r]\ddd\xi \ddd \xi'.\]
	Then Schur test implies that it suffices to bound the following term
	\[\sup_{\xi} \int_{\bR^d} \widehat{\phi}(\xi'-\xi) |\xi|^{1/2} |\xi'|^{1/2} \delta\l[\Phi_{\eta}(\xi)-\Phi_{\eta}(\xi')\r] \ddd\xi'\]
	independently of $\eta$. For every $\xi\in \bR^d$ and $\eta\in\bR^d$, denote the set
	\[\mathcal{V}_{\xi,\eta}:=\{\xi': |\xi'+\eta|^\alpha-\alpha|\eta|^{\alpha-2}\eta\xi'= |\xi+\eta|^\alpha-\alpha|\eta|^{\alpha-2}\eta\xi\}.\]
	Then a direct computation shows that the aforementioned term can be bounded by
	\begin{equation*}
		\sup_{\xi} \sup_{\xi'\in \mathcal{V}_{\xi,\eta}} \frac{\|\widehat{\phi}\|_{L^{\infty}} |\xi|^{1/2}|\xi'|^{1/2} |\eta|^{\alpha-2}}{\alpha \Big||\xi'+\eta|^{\alpha-2}(\xi'+\eta) -|\eta|^{\alpha-2}\eta\Big|}.
	\end{equation*}
	First, we states the following claim whose proof is postponed
	\begin{equation}\label{E:Approximate local smooth-1}
		\sup_{\xi} \sup_{\xi'\in \mathcal{V}_{\xi,\eta}} \frac{\|\widehat{\phi}\|_{L^{\infty}} |\xi|^{1/2}|\xi'|^{1/2} |\eta|^{\alpha-2}}{\alpha \Big||\xi'+\eta|^{\alpha-2}(\xi'+\eta) -|\eta|^{\alpha-2}\eta\Big|} \sim_{\alpha} \sup_{\xi} \sup_{\xi'\in \mathcal{V}_{\xi,\eta}} \frac{\|\widehat{\phi}\|_{L^{\infty}} |\xi|^{1/2}}{|\xi'|^{1/2}}.
	\end{equation}

	Now we turn to prove $|\xi'|\sim_{\alpha} |\xi|$ for all $\xi'\in \mathcal{V}_{\xi,\eta}$ uniformly in $\eta$. Then this estimate and the claim \eqref{E:Approximate local smooth-1} will immediately give the desired uniform bound $C_{\alpha}\|\widehat{\phi}\|_{L^{\infty}}$ and complete the proof. Indeed, the set $\mathcal{V}_{\xi,\eta}$ can be rewritten as
	\[\mathcal{V}_{\xi,\eta}=\l\{\xi': |\tilde{\xi}'+\bar{\eta}|^\alpha-\alpha|\tilde{\xi}'|\cos\theta(\xi',\eta) = |\tilde{\xi}+\bar{\eta}|^\alpha-\alpha|\tilde{\xi}|\cos\theta(\xi,\eta)\r\},\]
	where
	\[\tilde{\xi}':=\xi'/|\eta|, \quad \bar{\eta}:=\eta/|\eta|, \quad \tilde{\xi}:=\xi/|\eta|.\]
	Note that $|\tilde{\xi}|\leq 1/5$ and $|\tilde{\xi}'|\leq 1/5$. Without loss of generality, we may assume $\bar{\eta}=(1,0,\ldots,0)$. Then for $\xi'\in \mathcal{V}_{\xi,\eta}$ we have
	\begin{equation}\label{E:Approximate local smooth-2}
		(1+2\tilde{\xi}'_1+|\tilde{\xi}'|^2)^{\alpha/2} -\alpha\tilde{\xi}'_1 =(1+2\tilde{\xi}_1+|\tilde{\xi}|^2)^{\alpha/2} -\alpha\tilde{\xi}_1.
	\end{equation}
	Here we have used the notations $\tilde{\xi}=(\tilde{\xi}_1,\tilde{\xi}_2,\ldots,\tilde{\xi}_d)$ and $\tilde{\xi}' =(\tilde{\xi}'_1,\tilde{\xi}'_2,\ldots,\tilde{\xi}'_d)$. Moreover, we can use Taylor's theorem to deduce the estimates
	\[\l(1+2\tilde{\xi}_1+|\tilde{\xi}|^2\r)^{\alpha/2} -\alpha\tilde{\xi}_1 \leq \l(1+2|\tilde{\xi}|+|\tilde{\xi}|^2\r)^{\alpha/2} -\alpha|\tilde{\xi}|=1+\frac{\alpha(\alpha-1)}{2} \l(1+\theta_0 |\tilde{\xi}|\r)^{\alpha-2}|\tilde{\xi}|^2\]
	with $0<\theta_0<1$; and on the other hand there exists $0<\theta_1<1$ such that
	\[\l(1+2\tilde{\xi}'_1+|\tilde{\xi}'|^2\r)^{\alpha/2} -\alpha\tilde{\xi}'_1 =1+ \frac{\alpha}{2} |\tilde{\xi}'|^2+ \frac{\alpha(\alpha-1)}{2}\l[1+\theta_1 (2\tilde{\xi}'_1+|\tilde{\xi}'|^2)\r]^{\alpha-2} \l(|\tilde{\xi}'|+2\tilde{\xi}'_1\r)^2.\]
	Hence if there holds $|\tilde{\xi}| \ll |\tilde{\xi}'| \leq 1/5$, then it will deduce a contradiction to the identity \eqref{E:Approximate local smooth-2}. Similarly $|\xi'|\ll |\xi|$ is not possible. Therefore we obtain the desired estimate $|\xi|\sim_{\alpha} |\xi'|$ and it remains to prove the claim \eqref{E:Approximate local smooth-1}.
	
	First, we assume $\tilde{\xi}'_1=x$ and $|(\tilde{\xi}'_2,\ldots,\tilde{\xi}'_d)|=y$ for real numbers $x$ and $y$. Recall the aforementioned vector $\bar{\eta}$ and consider the vector $(x,y)\in \bR^2$ with $|(x,y)|\leq 1/5$. To prove the desired estimate \eqref{E:Approximate local smooth-1}, we only need to show that
	\[C_1\leq \frac{\Big|(x+1,y)-\l(1/|(x+1,y)|^{\alpha-2},0\r)\Big|}{\sqrt{x^2+y^2}} \leq C_2,\]
	which means
	\begin{equation}\label{E:Approximate local smooth-3}
		C_1^2\leq \frac{(x+1)^2+[(x+1)^2+y^2]^{2-\alpha}-2(x+1)[(x+1)^2+y^2]^{1-\frac{\alpha}{2}} +y^2}{x^2+y^2} \leq C_2^2.
	\end{equation}
	If $|(x,y)|\geq c_0$ for some $c_0>0$, the existence of $C_1$ and $C_2$ is obvious. Hence we should investigate the upper and lower bound for the case $|(x,y)|\to 0$. In this case, by denoting $z=x^2+y^2+2x$ and using Taylor's theorem, we conclude that the middle item in \eqref{E:Approximate local smooth-3} equals to
	\[\frac{x^2+y^2+(\alpha-2)x(x^2+y^2+2x)+O(z^2)-2(x+1)O(z^2)+O(z^{2\alpha-4})+O(z^{\alpha-2})}{x^2+y^2}.\]
	The limit does not exist as $(x,y)\to(0,0)$, but we can give the bounds for all points $(x,y)$ in the neighborhood of origin. Notice that for $|(x,y)|\leq \min\{3/(\alpha-2), 1/5\}$ there holds
	\[(\alpha-2)x(x^2+y^2+2x)\geq \frac{-x^2-y^2}{3}, \quad (\alpha-2)x(x^2+y^2+2x)\leq 3(\alpha-2)(x^2+y^2).\]
	Hence we obtain the desired lower bound $1/2$ and upper bound $3\alpha$ for the case $|(x,y)|<c_0$ with $c_0$ small enough. This further implies the desired conclusion \eqref{E:Approximate local smooth-3}.
\end{proof}

\begin{lemma}\label{L:Approximate local convergence}
	Let $(\xi_n)\subset \bR^d$ with $|\xi_n|\to \infty$. If a bounded sequence of functions $(f_n)$ in $L^2(\bR^d)$ satisfies
	\[\supp \widehat{f}_n \subset \{\xi: |\xi|\leq |\xi_n|/5\}, \quad f_n\rhu 0\]
	weakly in $L^2(\bR^d)$ as $n\to\infty$, then up to subsequences we have
	\[[T_{\alpha}^n]f_n\to0\]
	strongly in $L_{loc, t,x}^2(\bR^{d+1})$ and hence $[T_{\alpha}^n]f_n(t,x)\to0$ almost everywhere in $\bR^{d+1}$.
\end{lemma}
\begin{proof}[\textbf{Proof of Lemma \ref{L:Approximate local convergence}}]
	Considering the function $\psi_\eta(\xi)$ defined in Lemma \ref{L:Approximate local smooth}, it is not hard to see that for any given $\varepsilon>0$, there exists $\Lambda$ large enough depending only on $\varepsilon$ such that for all $|\xi|\geq \Lambda$ with $|\xi|\leq |\eta|/5$ there holds
	\[\frac{1}{\psi_\eta(\xi)}<\varepsilon.\]
	Hence the desired conclusion follows from a standard argument by imitating the proof of Lemma \ref{L:Local convergence}, which means controlling the $[P_{\Lambda}^{\bot}]$ term by using Lemma \ref{L:Approximate local smooth} and estimating the $[P_{\Lambda}]$ term by using dominated convergence theorem since we have the pointwise convergence \eqref{E:Asmptotic schrodinger-0.8}. Similar proof also can be found in \cite[Lemma 4.3]{FS2018}, and the details are omitted here for avoiding too much repetition.
\end{proof}

\section{Method of missing mass}\label{S:Method of Missing Mass}
In this section, by adapting the arguments in \cite[Section 2]{FLS2016} and \cite[Section 2]{FS2018}, we show how the missing mass method can give the desired precompactness result Theorem \ref{T:Precompactness}. First, we introduce some definitions.
\begin{definition}
	Let $(f_n)\subset L^2(\bR^d)$. We write $f_n\rsa 0$ if for all sequences of symmetries $([g_n])\subset G$ there holds the weak convergence $[g_n]f_n \rhu 0$ in $L^2(\bR^d)$.
\end{definition}
Define the set
\[\mathcal{P}:=\l\{(f_n): \|f_n\|_{L^2(\bR^d)}=1, f_n\rsa 0\r\}\]
and the \textit{sharp $\alpha$-Strichartz constant with respect to $\mathcal{P}$} as follows
\[\mathbf{M}^{*}_{d,\alpha}:=\sup \l\{\limsup_{n\to\infty} \l\|[E_{\alpha}]f_n\r\|_{L_{t,x}^{q_0}(\bR^{d+1})}: (f_n) \in \mathcal{P} \r\}.\]
As we consider the precompactness of extremal sequences for $\mathbf{M}_{d,\alpha}$, it is obvious that the sequences in $\mathcal{P}$ are not precompact up to symmetries and thus are our enemies. On the other hand, our next result Proposition \ref{P:Weak to strong} states that all the enemies are in $\mathcal{P}$.

\begin{prop}\label{P:Weak to strong}
	Let $d\geq 2$. All the extremal sequences for $\mathbf{M}_{d,\alpha}$ are precompact up to symmetries if and only if
	\[\mathbf{M}_{d,\alpha}>\mathbf{M}_{d,\alpha}^{*}.\]
\end{prop}
\begin{proof}[\textbf{Proof of Proposition \ref{P:Weak to strong}}]
	It is clear that $\mathbf{M}_{d,\alpha}\geq \mathbf{M}_{d,\alpha}^{*}$. Hence the `only if' part comes from the definition of $\mathbf{M}_{d,\alpha}^{*}$ and Theorem \ref{T:Mass} which claims that the supremum value $\mathbf{M}_{d,\alpha}^{*}$ can be attained.
	
	Thus, we only need to show the `if' part. Assume that $\mathbf{M}_{d,\alpha} > \mathbf{M}_{d,\alpha}^{*}$ and $(f_n)$ is an extremal sequence for $\mathbf{M}_{d,\alpha}$. In this case there exists $([g_n])\subset G$ such that, up to subsequences, $[g_n]f_n\rhu v\neq0$ weakly in $L^2(\bR^d)$ as $n\to\infty$ since $(f_n)\notin \mathcal{P}$. Hence if we write
	\[v_n:=[g_n]f_n, \quad r_n:=v_n-v,\]
	then for $n\to\infty$ we have $r_n\rhu0$ weakly in $L^2(\bR^d)$ and further
	\begin{equation}\label{E:Weak to strong-1}
		1=\|f_n\|_{L^2}^2=\|v_n\|_{L^2}^2=\|v\|_{L^2}^2+\|r_n\|_{L^2}^2+o(1)
	\end{equation}
	due to the fact that $L^2$ is a Hilbert space. Meanwhile the local convergence Lemma \ref{L:Local convergence} implies that $[E_{\alpha}]r_n(t,x)\to 0$ almost everywhere in $\bR^{d+1}$ since $q_0>2$. Then as $n\to\infty$, a variant of Br\'ezis-Lieb lemma \cite[Lemma 3.1]{FLS2016} gives that
	\begin{align*}
		\mathbf{M}_{d,\alpha}^{q_0} &=\|[E_{\alpha}]v_n\|_{L_{t,x}^{q_0}}^{q_0} +o(1) \\
		&= \|[E_{\alpha}]v\|_{L_{t,x}^{q_0}}^{q_0} + \|[E_{\alpha}]r_n\|_{L_{t,x}^{q_0}}^{q_0} +o(1) \\
		&\leq \mathbf{M}_{d,\alpha}^{q_0} \l(\|v\|_{L_x^2}^{q_0}+\|r_n\|_{L_x^2}^{q_0}\r) +o(1).
	\end{align*}
	Combining this estimate with \eqref{E:Weak to strong-1} and letting $n\to\infty$, we conclude
	\[1\leq \|v\|_{L^2}^{q_0}+\l(1-\|v\|_{L^2}^2\r)^{q_0/2}.\]
	Therefore the fact $q_0>2$ implies either $\|v\|_{L^2}^2=0$ or $1-\|v\|_{L^2}^2=0$. Since $v\neq 0$, we obtain that
	\[0=1-\|v\|_{L^2}^2=\lim_{n\to\infty} \|r_n\|_{L^2}^2 =\lim_{n\to\infty} \|v_n-v\|_{L^2}^2.\]
	This states that $v_n=[g_n]f_n$ convergence strongly to $v$ in $L^2(\bR^d)$ and completes the proof.
\end{proof}

The Proposition \ref{P:Weak to strong} above gives a characterization on the precompactness of extremal sequences for $\mathbf{M}_{d,\alpha}$. Hence our main result Theorem \ref{T:Precompactness} is reduced to the following Theorem \ref{T:Mass} which is also used in the proof of Proposition \ref{P:Weak to strong}.
\begin{thm}\label{T:Mass}
	Suppose that $\mathbf{a}_{d,\alpha}^{*}:=(\alpha-1)^{\frac{-1}{2d+4}} (\alpha/2)^{\frac{-d}{2d+4}}$ and $d\geq 2$. There holds
	\[\mathbf{M}^{*}_{d,\alpha}=\mathbf{a}_{d,\alpha}^{*} \mathbf{S}_d^{*}.\]
	Furthermore, the supremum $\mathbf{M}^{*}_{d,\alpha}$ is attained. In other words, there is a sequence $(f_n)\subset L^2(\bR^d)$ with $\|f_n\|_{L_x^2}=1$ such that $f_n\rsa 0$ and $\limsup_{n\to\infty} \|[E_{\alpha}]f_n\|_{L_{t,x}^{q_0}} =\mathbf{M}_{d,\alpha}^{*}$.
\end{thm}
\begin{proof}[\textbf{Proof of Theorem \ref{T:Mass}}]
	The proof follows from a more involved version of missing mass method. We first show that $\mathbf{M}_{d,\alpha}^{*}\geq \mathbf{a}_{d,\alpha}^{*} \mathbf{S}_{d}^{*}$. For a sequence $(\xi_n)$ with $|\xi_n|\to\infty$ and $\phi\in L^2(\bR^d)$ with $\|\phi\|_{L^2}=1$, we define $\widehat{f}_n(\xi):= \widehat{\phi}(\xi-\xi_n)$. It is not hard to see that $f_n\rsa 0$. Then Lemma \ref{L:Asymptotic schrodinger} implies that
	\[\lim_{n\to\infty} \l\|[E_{\alpha}]f_n\r\|_{L_{t,x}^{q_0}} =\mathbf{a}_{d,\alpha}^{*} \l\|[e^{it\Delta}]\phi\r\|_{L_{t,x}^{q_0}}.\]
	Hence we conclude
	\[\mathbf{M}_{d,\alpha}^{*}\geq \mathbf{a}_{d,\alpha}^{*} \l\|[e^{it\Delta}]\phi\r\|_{L_{t,x}^{q_0}}.\]
	By taking supremum over all such functions $\phi$, we obtain $\mathbf{M}_{d,\alpha}^{*}\geq \mathbf{a}_{d,\alpha}^{*} \mathbf{S}_{d}^{*}$.
	
	Furthermore, from \cite{Shao2009EJDE} it is known that there exists an extremal for $\mathbf{S}_{d}^{*}$. Then taking $\phi$ in the above argument to be this extremal will give the sequence $(f_n)$ with $\|f_n\|_{L^2}=1$ and $f_n\rsa 0$ such that
	\[\lim_{n\to\infty} \l\|[E_{\alpha}]f_n\r\|_{L_{t,x}^{q_0}} =\mathbf{a}_{d,\alpha}^{*} \mathbf{S}_{d}^{*}.\]
	Therefore, it remains to prove the reverse inequality $\mathbf{M}_{d,\alpha}^{*} \leq \mathbf{a}_{d,\alpha}^{*} \mathbf{S}_{d}^{*}$.
	
	We may assume $\mathbf{M}_{d,\alpha}^{*}>0$ without loss of generality. By the definition of $\mathbf{M}_{d,\alpha}^{*}$, there exists $(f_n)\subset L^2(\bR^d)$ satisfying $\|f_n\|_{L^2}=1$ and $f_n\rsa 0$ such that
	\[\limsup_{n\to\infty} \l\|[E_{\alpha}] f_n\r\|_{L_{t,x}^{q_0}}\geq \frac{1}{2}\mathbf{M}_{d,\alpha}^{*} >0.\]
	Then Corollary \ref{C:Non-zero weak limit} yields that there exist $([g_n]) \subset G$ and $\xi_n$ with $h_n |\xi_n|\geq \frac{1}{2}$ such that up to subsequences there holds
	\[\widehat{[g_n]f_n}\l(\xi+h_n\xi_n\r) \rhu \widehat{V}(\xi), \quad \|V\|_{L^2}\geq \tilde{\gamma}>0.\]
	Here $\tilde{\gamma}:=\tilde{\gamma}(d,\alpha)$ depends only on the dimension $d$ and $\alpha$. Note that $f_n\rsa 0$ implies $h_n|\xi_n|\to \infty$. Otherwise we have $\widehat{V}=0$ which is a contradiction. Hence if we set $E_n:=\{\xi\in \bR^d: |\xi|\leq h_n|\xi_n|/5\}$ and write 
	\[\widehat{r}_n:=\mathds{1}_{E_n} \l(\widehat{[g_n]f_n}(\cdot + h_n\xi_n) -\widehat{V}\r),\quad \widehat{q}_n:=(\mathds{1}_{\bR^d}-\mathds{1}_{E_n}) \l(\widehat{[g_n]f_n}(\cdot + h_n\xi_n) -\widehat{V}\r),\]
	then there holds
	\begin{equation*}
		\l\|[E_{\alpha}] f_n\r\|_{L_{t,x}^{q_0}} = \l\|[E_{\alpha}][g_n] f_n\r\|_{L_{t,x}^{q_0}} =\|[T_{\alpha}^n](r_n+q_n+V) \|_{L_{t,x}^{q_0}}= \|[\bar{T}_{\alpha}^n](r_n+q_n+V) \|_{L_{t,x}^{q_0}},
	\end{equation*}
	where $[T_{\alpha}^n]$ and $[\bar{T}_{\alpha}^n]$ are defined as
	\begin{equation*}
		[T_{\alpha}^n]f(t,x):=\int_{\bR^d} |\xi+h_n \xi_n|^{\frac{\alpha-2}{q_0}} e^{ix\xi +it|\xi+h_n\xi_n|^{\alpha}} \widehat{f}(\xi) \ddd \xi
	\end{equation*}
	and
	\begin{equation*}
		[\bar{T}_{\alpha}^n]f(t,x):=\int_{\bR^d} \l|\frac{\xi}{h_n|\xi_n|}+\bar{\xi}_n\r|^{\frac{\alpha-2}{q_0}} e^{ix\xi+it\Phi_n(\xi)} \widehat{f}(\xi) \ddd\xi.
	\end{equation*}
	Here $\bar{\xi_n}:= \xi_n/|\xi_n|$ and the function $\Phi_n(\xi)$ is defined by
	\[\Phi_n(\xi):=\frac{1}{|h_n \xi_n|^{\alpha-2}} \l[|\xi+h_n \xi_n|^{\alpha}-|h_n \xi_n|^{\alpha}- \alpha|h_n \xi_n|^{\alpha-2} h_n\xi_n\xi\r].\]
	Notice that $q_n\to0$ strongly in $L^2(\bR^d)$ and further $r_n\rhu 0$ weakly in $L^2(\bR^d)$. Hence the uniform boundedness of the operators $[\bar{T}_{\alpha}^n]$ as shown in \eqref{E:Asmptotic schrodinger-0.5} implies that $[\bar{T}_{\alpha}^n]q_n\to 0$ strongly in $L_{t,x}^{q_0}(\bR^{d+1})$. Meanwhile, by the local convergence Lemma \ref{L:Approximate local convergence}, we conclude $[\bar{T}_{\alpha}^n] r_n \to 0$ almost everywhere in $\bR^{d+1}$. Note that the strong convergence result \eqref{E:Asymptotic schrodinger-2} implies
	\[\lim_{n\to\infty} \l\|[\bar{T}_{\alpha}^n] V(t,x) - \int_{\bR^d} e^{ix\xi+it\l[\alpha |\xi|^2/2 +\alpha(\alpha-2)|\xi\xi_0|^2/2\r]} \widehat{V}(\xi) \ddd \xi\r\|_{L_{t,x}^{q_0}} =0.\]
	Thus as $n\to\infty$, applying the variant of Br\'{e}zis-Lieb lemma \cite[Lemma 3.1]{FLS2016} with
	\[\pi_n=\int_{\bR^d} e^{ix\xi+it\l[\alpha |\xi|^2/2 +\alpha(\alpha-2)|\xi\xi_0|^2/2\r]} \widehat{V}(\xi) \ddd \xi, \quad \rho_n=[\bar{T}_{\alpha}^n] r_n,\]
	we can obtain that
	\begin{equation} \label{E:Mass-1}
		\l\|[E_{\alpha}] f_n\r\|_{L_{t,x}^{q_0}}^{q_0} = \l\|[\bar{T}_{\alpha}^n] V\r\|_{L_{t,x}^{q_0}}^{q_0}+ \|[\bar{T}_{\alpha}^n]r_n \|_{L_{t,x}^{q_0}}^{q_0} +o(1) =\l\|[T_{\alpha}^n] V\r\|_{L_{t,x}^{q_0}}^{q_0}+ \|[T_{\alpha}^n]r_n \|_{L_{t,x}^{q_0}}^{q_0} +o(1).
	\end{equation}
	Due to the asymptotic Schr\"odinger behavior Lemma \ref{L:Asymptotic schrodinger} and Strichartz inequality \eqref{E:alpha-Strichartz}, we obtain the following asymptotic estimate
	\[\lim_{n\to\infty} \l\|[T_{\alpha}^n]V\r\|_{L_{t,x}^{q_0}} =\mathbf{a}_{d,\alpha}^{*} \l\|[e^{it\Delta}][\tilde{A}_0]V\r\|_{L_{t,x}^{q_0}} \leq \mathbf{a}_{d,\alpha}^{*} \mathbf{S}_d^{*} \|V\|_{L_x^2}.\]
	Here $[\tilde{A}_0]$ is an unitary operator on $L_x^2$ defined in \eqref{E:Approximate unitary operator}. On the other hand, a changing of variables deduces $[T_{\alpha}^n]r_n=[E_{\alpha}]\omega_n$ with $\widehat{\omega}_n:=\widehat{r}_n(\cdot-h_n\xi_n)$. Since $\widehat{[g_n]f_n}\rsa0$ and $\widehat{V}(\cdot -h_n\xi_n) \rsa 0$, as well as the Fourier transform is an automorphism operator on $L^2$, we obtain that $\omega_n\rsa 0$. This fact implies
	\[\limsup_{n\to\infty} \l\|[T_{\alpha}^n]r_n\r\|_{L_{t,x}^{q_0}} =\limsup_{n\to\infty} \|[E_{\alpha}]\omega_n\|_{L_{t,x}^{q_0}} \leq \mathbf{M}_{d,\alpha}^{*} \l(1-\|V\|_{L^2}^2\r)^{1/2}.\]
	Inserting these two limit estimates into \eqref{E:Mass-1} and taking limit $n\to \infty$ can deduce the following inequality
	\[\limsup_{n\to\infty} \|[E_{\alpha}] f_n\|_{L_{t,x}^{q_0}}^{q_0} \leq \l(\mathbf{a}_{d,\alpha}^{*} \mathbf{S}_d^{*}\r)^{q_0} \|V\|_{L^2}^{q_0} + \l(\mathbf{M}_{d,\alpha}^{*}\r)^{q_0} \l(1-\|V\|_{L^2}^2\r)^{q_0/2};\]
	and this inequality can be rewritten as
	\[\l(\mathbf{M}_{d,\alpha}^{*}\r)^{q_0} \l[1-\l(1-\|V\|_{L^2}^2\r)^{\frac{q_0}{2}}\r]- \l(\mathbf{a}_{d,\alpha}^{*} \mathbf{S}_d^{*}\r)^{q_0} \|V\|_{L^2}^{q_0} \leq \l(\mathbf{M}_{d,\alpha}^{*}\r)^{q_0}- \limsup_{n\to\infty} \|[E_{\alpha}] f_n\|_{L_{t,x}^{q_0}}^{q_0}.\]
	Since $q_0>2$ and $\|V\|_{L^2}\in [\tilde{\gamma},1]$, we conclude
	\[\l[\l(\mathbf{M}_{d,\alpha}^{*}\r)^{q_0}- \l(\mathbf{a}_{d,\alpha}^{*} \mathbf{S}_d^{*}\r)^{q_0}\r] \tilde{\gamma}^{q_0} \leq \l(\mathbf{M}_{d,\alpha}^{*}\r)^{q_0}- \limsup_{n\to\infty} \|[E_{\alpha}] f_n\|_{L_{t,x}^{q_0}}^{q_0}.\]
	Finally, taking supremum over all such sequences $(f_n)$ yields the desired conclusion.
\end{proof}

\section{Existence of extremals} \label{S:Existence of extremals}
Based on the precompactness Theorem \ref{T:Precompactness} established in the previous section, one direct way to show the existence of extremals for $\mathbf{M}_{d,\alpha}$ is to compare this sharp constant $\mathbf{M}_{d,\alpha}$ with the constant $\mathbf{S}_d^{*}$. In fact, as mentioned in the introduction, several works have been done in this direction, and similar strict-inequality phenomena have been appeared in other surfaces such as the sphere, see for example the articles \cite{BOQ2020,CS2012A&P,FLS2016,OQ2018,Shao2016}.

For our fractional surface situation, with the dimension $d=2$ and the index $\alpha$ in some region, Oliveira e Silva and Quilodr\'an \cite[Proposition 6.9]{OQ2018} have established the desired strict inequality \eqref{T:Precompactness-1} by applying some comparison principle for convolutions of certain singular measures. This result helps us obtain the existence of extremals for $\mathbf{M}_{2,\alpha}$ with the index $\alpha$ in this corresponding region. For the convenience of the reader, we recall the result \cite[Proposition 6.9]{OQ2018} in this article and then show the detailed proof for Theorem \ref{T:Existence}.

Here we recall the strict-inequality result of Oliveira e Silva and Quilodr\'an \cite[Proposition 6.9]{OQ2018}. For $d=2$, they have investigated the following convolution inequality
\[\l\|f\sigma_{\alpha} \ast f\sigma_{\alpha} \r\|_{L_{t,x}^2(\bR^3)} \leq \mathcal{Q}_{\alpha}^2 \|f\|_{L_x^2(\bR^2)},\]
where $\mathcal{Q}_{\alpha}$ is the sharp constant and the singular measure $\sigma_{\alpha}$ is given by
\[\ddd \sigma_{\alpha}(y,s):= \delta(s-|y|^{\alpha})|y|^{\frac{\alpha-2}{4}} \ddd y \ddd s, \quad (y,s)\in\bR^2\times \bR.\]
Then they have established the following proposition.
\begin{prop}[Proposition 6.9 of \cite{OQ2018}] \label{P:Strict inequaltiy}
	For the dimension $d=2$, there exists a constant $\alpha_0> 5$ such that for arbitrary $\alpha \in (2,\alpha_0)$ there holds
	\[\frac{\pi}{\alpha\sqrt{\alpha-1}} < \mathcal{Q}_{\alpha}^4 \leq \frac{\pi}{\alpha}.\]
\end{prop}
Finally, we show the desired existence of extremals Theorem \ref{T:Existence} and complete this section.
\begin{proof}[\textbf{Proof of Theorem \ref{T:Existence}}]
	Let us first recall the classical result that up to symmetries Gaussians are the only extremals for $\mathbf{S}_d^{*}$ when $d=\{1,2\}$. Hence, by direct computation, one can obtain the following sharp constants
	\[\mathbf{S}_1^{*}= 12^{-1/12},\quad \mathbf{S}_2^{*}= 2^{-1/2}.\]
	Then applying the precompactness Theorem \ref{T:Precompactness}, we obtain that the extremals for $\mathbf{M}_{2,\alpha}$ must exist if we can show the following strict inequality
	\begin{equation} \label{E:Existence-1}
		\mathbf{M}_{2,\alpha}^4 > (2\alpha\sqrt{\alpha-1})^{-1}.
	\end{equation}
	Denoting the space-time Fourier transform by $\mathscr{F}_{t,x}$ with dimension $d=2$, one can observe that
	\[[D^{\frac{\alpha-2}{4}}][e^{it|\nabla|^{\alpha}}] \mathscr{F}^{-1} f(t,x)= (2\pi)^{-2} \mathscr{F}_{t,x} (f\sigma_{\alpha}) (-t,-x).\]
	Hence for the sharp constants $\mathbf{M}_{2,\alpha}$ and $\mathcal{Q}_{\alpha}$, by the Plancherel theorem, there holds
	\[(2\pi)\mathbf{M}_{2,\alpha}^4 = \mathcal{Q}_{\alpha}^4.\]
	Then applying the aforementioned Proposition \ref{P:Strict inequaltiy}, we obtain that the desired strict inequality \eqref{E:Existence-1} holds for $\alpha\in (2,\alpha_0)$ with some index $\alpha_0>5$. This completes the proof.
\end{proof}

\bigskip
\subsection*{Acknowledgments}
The authors would like to thank Ren\'{e} Quilodr\'{a}n for bring their attention to the result \cite[Proposition 6.9]{OQ2018}, and thank Shuanglin Shao for valuable conversations. The first author acknowledges the support from University of Chinese Academy of Sciences Joint Training Program. And this work is completed during the first author's visit to the University of Kansas whose hospitality is also appreciated.

\bigskip
\appendix
\numberwithin{equation}{section}
\renewcommand{\theequation}{\thesection.\arabic{equation}}

\section{One geometric consequence}\label{S:A geometric result}
\noindent For a parameter $\eta$, we define the following function
\[F_{\eta}(\xi)=|\xi|^{\alpha}+|\eta|^{\alpha}-|\xi+\eta|^{\alpha}/2^{\alpha-1}\]
on the set $E_{\eta}:=\{\xi\in\bR^d: |\xi|\geq |\eta|\}$. This appendix is devoted to showing a geometric result related to this function. Roughly speaking, the Proposition \ref{P:Angle decomposition} below states that if the angle between $\xi$ and $\eta$ is small, then the Hessian matrix of $F_{\eta}(\xi)$ is positive-definite which means the corresponding surface has positive Gaussian curvature. Hence, in this situation, the multi-variable Taylor's theorem will give some nice displacement estimates around the critical points. These estimates are crucial when we establish the quasi-orthogonality Lemma \ref{L:Quasi-orthogonality} and apply the bilinear restriction theory. We first introduce some notations.

For two positive semi-definite matrices $A$ and $B$, we write $A\geq B$ if $A-B\geq0$ which means $A-B$ is positive semi-definite. Similarly for $A>B$, $A\leq B$ and $A<B$. Meanwhile the vectors in $\bR^d$ are viewed as row vectors and the notation $\xi^T$ denotes the transpose of $\xi$. For two vectors $\xi$ and $\eta$ in $\bR^d$, we use the notation $\theta(\xi,\eta)$ to denote the angle between $\xi$ and $\eta$. In addition, $E_d$ is the $d\times d$ identity matrix. Then our result is as follows.
\begin{prop}\label{P:Angle decomposition}
	There exists a number $\bar{K}_{d,\alpha}\in\bZ_{+}$ such that the unit cube-annular $\mathcal{A}_1$ can be divided into $\bar{K}_{d,\alpha}$ parts $\mathcal{A}_1=\bigcup_{j=1}^{\bar{K}_{d,\alpha}} \mathcal{R}_j$, and in each part $\mathcal{R}_j$ we have the following properties:
	\[c_1 E_d\leq \mathrm{Hess}F_{\eta}(\xi) \leq c_2 E_d\]
	holds for all pairs $(\xi,\eta)\in \mathcal{R}_j\times \mathcal{R}_j$ with $\xi \in E_{\eta}$.
\end{prop}
\begin{proof}[\textbf{Proof of Proposition \ref{P:Angle decomposition}}]
	Firstly it is not hard to see that, by dividing the unit sphere $\bS^{d-1}$ into $\bar{K}_{d,\alpha}$ disjoint parts and then decomposing $\mathcal{A}_1$ accordingly, we can achieve that for every $(\xi,\eta)\in \mathcal{R}_j\times \mathcal{R}_j$ and every $y\in\bR^d$ there holds
	\begin{equation}\label{E:Angle decomposition-1}
		\l|\cos^2\theta(\xi,y)-\cos^2\theta(\xi+\eta,y)\r|\leq \frac{1}{2\alpha-4}.
	\end{equation}
	Indeed, since $|\theta(\xi+\eta,y)-\theta(\xi,y)|\leq \theta(\xi+\eta,\eta) \leq \theta(\xi,\eta)$, the continuity of trigonometric functions gives the existence of such number $\bar{K}_{d,\alpha}$. 
	
	A direct computation shows
	\[\mathrm{Hess} F_{\eta}(\xi)=\alpha\l[|\xi|^{\alpha-2}-\frac{|\xi+\eta|^{\alpha-2}}{2^{\alpha-1}}\r]E_d +\alpha(\alpha-2)\l[|\xi|^{\alpha-4}\xi^T\xi-\frac{|\xi+\eta|^{\alpha-4}(\xi+\eta)^T(\xi+\eta)}{2^{\alpha-1}}\r].\]
	Therefore the inequality $|\xi+\eta|^{\alpha-2}\leq (|\xi|+|\eta|)^{\alpha-2} \leq 2^{\alpha-2}|\xi|^{\alpha-2}$ implies the following estimate
	\[\mathrm{Hess} F_{\eta}(\xi) \geq \frac{\alpha|\xi|^{\alpha-2}}{2}\l[E_d+(\alpha-2)\frac{\xi^T\xi}{|\xi|^2}-(\alpha-2) \frac{(\xi+\eta)^T(\xi+\eta)}{|\xi+\eta|^2}\r].\]
	We aim to show that for every $y\in\bR^d$ there holds
	\[c_3|y|^2\leq y\l[E_d+(\alpha-2)\frac{\xi^T\xi}{|\xi|^2}-(\alpha-2) \frac{(\xi+\eta)^T(\xi+\eta)}{|\xi+\eta|^2}\r]y^T,\]
	where $c_3$ is independent of $y$. This fact will give the desired constant $c_1$ since $|\xi|\sim 1$ in our situation. By homogeneity we may assume $y\in\bS^{d-1}$. Then we conclude that
	\begin{align*}
		&y\l[E_d+(\alpha-2)\frac{\xi^T\xi}{|\xi|^2}-(\alpha-2) \frac{(\xi+\eta)^T(\xi+\eta)}{|\xi+\eta|^2}\r]y^T \\
		&=1+(\alpha-2)\l[\frac{|\xi y^T|^2}{|\xi|^2}-\frac{|(\xi+\eta)y^T|^2}{|\xi+\eta|^2}\r] \\
		&\geq 1-(\alpha-2)\l|\cos^2\theta(\xi,y)-\cos^2\theta(\xi+\eta,y)\r| \\
		&\geq 1/2,
	\end{align*}
	where in the last inequality we have used the condition \eqref{E:Angle decomposition-1}. Finally, since $F_{\eta}(\xi)$ is smooth and the domain is bounded, the existence of $c_2$ is obvious.
\end{proof}
\begin{remark}\label{R:Angle decomposition}
	By the proof, we can define
	\[\bar{\theta}_j:=\sup\{\theta(\xi,\eta): (\xi,\eta)\in\mathcal{R}_j\times\mathcal{R}_j, |\xi|\geq |\eta|\}, \quad \bar{\theta}_0:=\min\{\bar{\theta}_j: j=1,2,\ldots, \bar{K}_{d,\alpha}\}.\]
	Using these notations, Proposition \ref{P:Angle decomposition} implies that if $\theta(\xi,\eta)\leq \bar{\theta}_0$ with $|\xi|\geq |\eta|$, then there holds
	\[c_1 E_d \leq\mathrm{Hess} F_{\eta}(\xi)\leq c_2 E_d.\]
\end{remark}


\begin{thebibliography}{10}
	
	\bibitem{ABZ2017}
	Thomas Alazard, Nicolas Burq, and Claude Zuily.
	\newblock A stationary phase type estimate.
	\newblock {\em Proc. Amer. Math. Soc.}, 145(7):2871--2880, 2017.
	\newblock \url{https://doi.org/10.1090/proc/13199}.
	
	\bibitem{BG1999}
	Hajer Bahouri and Patrick G\'erard.
	\newblock High frequency approximation of solutions to critical nonlinear wave
	equations.
	\newblock {\em Amer. J. Math.}, 121(1):131--175, 1999.
	\newblock \url{https://doi.org/10.1353/ajm.1999.0001}.
	
	\bibitem{BV2007}
	Pascal B\'egout and Ana Vargas.
	\newblock Mass concentration phenomena for the {$L^2$}-critical nonlinear
	{S}chr\"odinger equation.
	\newblock {\em Trans. Amer. Math. Soc.}, 359(11):5257--5282, 2007.
	\newblock \url{https://doi.org/10.1090/S0002-9947-07-04250-X}.
	
	\bibitem{BS2023}
	Chandan Biswas and Betsy Stovall.
	\newblock Existence of extremizers for {F}ourier restriction to the moment
	curve.
	\newblock {\em Trans. Amer. Math. Soc.}, 376(5):3473--3492, 2023.
	\newblock \url{https://doi.org/10.1090/tran/8872}.
	
	\bibitem{BL1983}
	Ha\"{\i}m Br\'ezis and Elliott~H. Lieb.
	\newblock A relation between pointwise convergence of functions and convergence
	of functionals.
	\newblock {\em Proc. Amer. Math. Soc.}, 88(3):486--490, 1983.
	\newblock \url{https://doi.org/10.2307/2044999}.
	
	\bibitem{BOQ2020}
	Gianmarco Brocchi, Diogo Oliveira~e Silva, and Ren\'e Quilodr\'an.
	\newblock Sharp {S}trichartz inequalities for fractional and higher-order
	{S}chr\"odinger equations.
	\newblock {\em Anal. PDE}, 13(2):477--526, 2020.
	\newblock \url{https://doi.org/10.2140/apde.2020.13.477}.
	
	\bibitem{CK2007}
	R\'emi Carles and Sahbi Keraani.
	\newblock On the role of quadratic oscillations in nonlinear {S}chr\"odinger
	equations. {II}. {T}he {$L^2$}-critical case.
	\newblock {\em Trans. Amer. Math. Soc.}, 359(1):33--62, 2007.
	\newblock \url{https://doi.org/10.1090/S0002-9947-06-03955-9}.
	
	\bibitem{Carneiro2009}
	Emanuel Carneiro.
	\newblock A sharp inequality for the {S}trichartz norm.
	\newblock {\em Int. Math. Res. Not. IMRN}, 2009(16):3127--3145, 2009.
	\newblock \url{https://doi.org/10.1093/imrn/rnp045}.
	
	\bibitem{COS2019ANIHPC}
	Emanuel Carneiro, Diogo Oliveira~e Silva, and Mateus Sousa.
	\newblock Extremizers for {F}ourier restriction on hyperboloids.
	\newblock {\em Ann. Inst. H. Poincar\'e C Anal. Non Lin\'eaire},
	36(2):389--415, 2019.
	\newblock \url{https://doi.org/10.1016/j.anihpc.2018.06.001}.
	
	\bibitem{COSS2021}
	Emanuel Carneiro, Diogo Oliveira~e Silva, Mateus Sousa, and Betsy Stovall.
	\newblock Extremizers for adjoint {F}ourier restriction on hyperboloids: the
	higher-dimensional case.
	\newblock {\em Indiana Univ. Math. J.}, 70(2):535--559, 2021.
	\newblock \url{https://doi.org/10.1512/iumj.2021.70.8323}.
	
	\bibitem{CQ2014}
	Michael Christ and Ren\'e Quilodr\'an.
	\newblock Gaussians rarely extremize adjoint {F}ourier restriction inequalities
	for paraboloids.
	\newblock {\em Proc. Amer. Math. Soc.}, 142(3):887--896, 2014.
	\newblock \url{https://doi.org/10.1090/S0002-9939-2013-11827-7}.
	
	\bibitem{CS2012A&P}
	Michael Christ and Shuanglin Shao.
	\newblock Existence of extremals for a {F}ourier restriction inequality.
	\newblock {\em Anal. PDE}, 5(2):261--312, 2012.
	\newblock \url{https://doi.org/10.2140/apde.2012.5.261}.
	
	\bibitem{DY2023}
	Boning Di and Dunyan Yan.
	\newblock Extremals for {$\alpha$}-{S}trichartz inequalities.
	\newblock {\em J. Geom. Anal.}, 33(4):136, 2023.
	\newblock \url{https://doi.org/10.1007/s12220-022-01185-7}.
	
	\bibitem{FVV2011}
	Luca Fanelli, Luis Vega, and Nicola Visciglia.
	\newblock On the existence of maximizers for a family of restriction theorems.
	\newblock {\em Bull. Lond. Math. Soc.}, 43(4):811--817, 2011.
	\newblock \url{https://doi.org/10.1112/blms/bdr014}.
	
	\bibitem{FVV2012}
	Luca Fanelli, Luis Vega, and Nicola Visciglia.
	\newblock Existence of maximizers for {S}obolev-{S}trichartz inequalities.
	\newblock {\em Adv. Math.}, 229(3):1912--1923, 2012.
	\newblock \url{https://doi.org/10.1016/j.aim.2011.12.012}.
	
	\bibitem{FS2022}
	Taryn~C. Flock and Betsy Stovall.
	\newblock On extremizing sequences for adjoint {F}ourier restriction to the
	sphere.
	\newblock {\em arXiv:2204.10361}, 2022.
	\newblock \url{https://doi.org/10.48550/arXiv.2204.10361}.
	
	\bibitem{Foschi2007}
	Damiano Foschi.
	\newblock Maximizers for the {S}trichartz inequality.
	\newblock {\em J. Eur. Math. Soc. (JEMS)}, 9(4):739--774, 2007.
	\newblock \url{https://doi.org/10.4171/JEMS/95}.
	
	\bibitem{Foschi2015}
	Damiano Foschi.
	\newblock Global maximizers for the sphere adjoint {F}ourier restriction
	inequality.
	\newblock {\em J. Funct. Anal.}, 268(3):690--702, 2015.
	\newblock \url{https://doi.org/10.1016/j.jfa.2014.10.015}.
	
	\bibitem{FO2017}
	Damiano Foschi and Diogo Oliveira~e Silva.
	\newblock Some recent progress on sharp {F}ourier restriction theory.
	\newblock {\em Anal. Math.}, 43(2):241--265, 2017.
	\newblock \url{https://doi.org/10.1007/s10476-017-0306-2}.
	
	\bibitem{FLS2016}
	Rupert~L. Frank, Elliott~H. Lieb, and Julien Sabin.
	\newblock Maximizers for the {S}tein-{T}omas inequality.
	\newblock {\em Geom. Funct. Anal.}, 26(4):1095--1134, 2016.
	\newblock \url{https://doi.org/10.1007/s00208-018-1695-7}.
	
	\bibitem{FS2018}
	Rupert~L. Frank and Julien Sabin.
	\newblock Extremizers for the {A}iry-{S}trichartz inequality.
	\newblock {\em Math. Ann.}, 372(3-4):1121--1166, 2018.
	\newblock \url{https://doi.org/10.1007/s00039-016-0380-9}.
	
	\bibitem{HS2012}
	Dirk Hundertmark and Shuanglin Shao.
	\newblock Analyticity of extremizers to the {A}iry-{S}trichartz inequality.
	\newblock {\em Bull. Lond. Math. Soc.}, 44(2):336--352, 2012.
	\newblock \url{https://doi.org/10.1112/blms/bdr098}.
	
	\bibitem{HZ2006}
	Dirk Hundertmark and Vadim Zharnitsky.
	\newblock On sharp {S}trichartz inequalities in low dimensions.
	\newblock {\em Int. Math. Res. Not.}, 2006:1--18, 2006.
	\newblock \url{https://doi.org/10.1155/IMRN/2006/34080}.
	
	\bibitem{JPS2010}
	Jin-Cheng Jiang, Benoit Pausader, and Shuanglin Shao.
	\newblock The linear profile decomposition for the fourth order {S}chr\"odinger
	equation.
	\newblock {\em J. Differential Equations}, 249(10):2521--2547, 2010.
	\newblock \url{https://doi.org/10.1016/j.jde.2010.06.014}.
	
	\bibitem{JSS2017}
	Jin-Cheng Jiang, Shuanglin Shao, and Betsy Stovall.
	\newblock Linear profile decompositions for a family of fourth order
	schrödinger equations.
	\newblock {\em arXiv:1410.7520}, 2017.
	\newblock \url{https://doi.org/10.48550/arXiv.1410.7520}.
	
	\bibitem{KPV1991}
	Carlos~E. Kenig, Gustavo Ponce, and Luis Vega.
	\newblock Oscillatory integrals and regularity of dispersive equations.
	\newblock {\em Indiana Univ. Math. J.}, 40(1):33--69, 1991.
	\newblock \url{https://doi.org/10.1512/iumj.1991.40.40003}.
	
	\bibitem{Keraani2001}
	Sahbi Keraani.
	\newblock On the defect of compactness for the {S}trichartz estimates of the
	{S}chr\"odinger equations.
	\newblock {\em J. Differential Equations}, 175(2):353--392, 2001.
	\newblock \url{https://doi.org/10.1006/jdeq.2000.3951}.
	
	\bibitem{KSV2012}
	Rowan Killip, Betsy Stovall, and Monica Visan.
	\newblock Scattering for the cubic {K}lein-{G}ordon equation in two space
	dimensions.
	\newblock {\em Trans. Amer. Math. Soc.}, 364(3):1571--1631, 2012.
	\newblock \url{https://doi.org/10.1090/S0002-9947-2011-05536-4}.
	
	\bibitem{KV2013}
	Rowan Killip and Monica Visan.
	\newblock Nonlinear {S}chr\"odinger equations at critical regularity.
	\newblock In {\em Evolution equations}, volume~17 of {\em Clay Math. Proc.},
	pages 325--437. Amer. Math. Soc., Providence, RI, 2013.
	\newblock \url{https://www.claymath.org/library/proceedings/cmip017c.pdf#page=333}.
	
	\bibitem{Kunze2003}
	Markus Kunze.
	\newblock On the existence of a maximizer for the {S}trichartz inequality.
	\newblock {\em Comm. Math. Phys.}, 243(1):137--162, 2003.
	\newblock \url{https://doi.org/10.1007/s00220-003-0959-5}.
	
	\bibitem{Lieb1983}
	Elliott~H. Lieb.
	\newblock Sharp constants in the {H}ardy-{L}ittlewood-{S}obolev and related
	inequalities.
	\newblock {\em Ann. of Math. (2)}, 118(2):349--374, 1983.
	\newblock \url{https://doi.org/10.2307/2007032}.
	
	\bibitem{Lions1984a}
	Pierre-Louis Lions.
	\newblock The concentration-compactness principle in the calculus of
	variations. {T}he locally compact case. {I}.
	\newblock {\em Ann. Inst. H. Poincar\'e Anal. Non Lin\'eaire}, 1(2):109--145,
	1984.
	\newblock \url{http://doi.org/10.1016/S0294-1449(16)30428-0}.
	
	\bibitem{Lions1984b}
	Pierre-Louis Lions.
	\newblock The concentration-compactness principle in the calculus of
	variations. {T}he locally compact case. {II}.
	\newblock {\em Ann. Inst. H. Poincar\'e Anal. Non Lin\'eaire}, 1(4):223--283,
	1984.
	\newblock \url{http://doi.org/10.1016/S0294-1449(16)30422-X}.
	
	\bibitem{MV1998}
	Frank Merle and Luis Vega.
	\newblock Compactness at blow-up time for {$L^2$} solutions of the critical
	nonlinear {S}chr\"odinger equation in 2{D}.
	\newblock {\em Internat. Math. Res. Notices}, 8:399--425, 1998.
	\newblock \url{https://doi.org/10.1155/S1073792898000270}.
	
	\bibitem{NOT2022}
	Giuseppe Negro, Diogo Oliveira~e Silva, and Christoph Thiele.
	\newblock When does {$e^{-|\tau|}$} maximize {F}ourier extension for a conic
	section?
	\newblock In {\em Harmonic Analysis and Convexity}, volume~9 of {\em Adv. Anal.
		Geom.}, pages 391--426. De Gruyter, Berlin, 2023.
	\newblock \url{https://doi.org/10.1515/9783110775389-009}.
	
	\bibitem{OQ2018}
	Diogo Oliveira~e Silva and Ren\'e Quilodr\'an.
	\newblock On extremizers for {S}trichartz estimates for higher order
	{S}chr\"odinger equations.
	\newblock {\em Trans. Amer. Math. Soc.}, 370(10):6871--6907, 2018.
	\newblock \url{https://doi.org/10.1090/tran/7223}.
	
	\bibitem{Ramos2012}
	Javier Ramos.
	\newblock A refinement of the {S}trichartz inequality for the wave equation
	with applications.
	\newblock {\em Adv. Math.}, 230(2):649--698, 2012.
	\newblock \url{https://doi.org/10.1016/j.aim.2012.02.020}.
	
	\bibitem{Shao2009}
	Shuanglin Shao.
	\newblock The linear profile decomposition for the airy equation and the
	existence of maximizers for the airy {S}trichartz inequality.
	\newblock {\em Anal. PDE}, 2(1):83--117, 2009.
	\newblock \url{https://doi.org/10.2140/apde.2009.2.83}.
	
	\bibitem{Shao2009EJDE}
	Shuanglin Shao.
	\newblock Maximizers for the {S}trichartz and the {S}obolev-{S}trichartz
	inequalities for the {S}chr\"odinger equation.
	\newblock {\em Electron. J. Differential Equations}, 2009(3):1--13, 2009.
	\newblock \url{https://ejde.math.txstate.edu/Volumes/2009/03/shao.pdf}.
	
	\bibitem{Shao2016}
	Shuanglin Shao.
	\newblock On existence of extremizers for the {T}omas-{S}tein inequality for
	$\mathbb{S}^1$.
	\newblock {\em J. Funct. Anal.}, 270(10):3996--4038, 2016.
	\newblock \url{https://doi.org/10.1016/j.jfa.2016.02.019}.
	
	\bibitem{Stein1993}
	Elias~M. Stein.
	\newblock {\em Harmonic Analysis: Real-variable Methods, Orthogonality, and
		Oscillatory Integrals}, volume~43 of {\em Princeton Mathematical Series}.
	\newblock Princeton University Press, Princeton, NJ, 1993.
	\newblock With the assistance of Timothy S. Murphy, Monographs in Harmonic
	Analysis, III.
	
	\bibitem{Stovall2020}
	Betsy Stovall.
	\newblock Extremizability of {F}ourier restriction to the paraboloid.
	\newblock {\em Adv. Math.}, 360:1--18, 2020.
	\newblock \url{https://doi.org/10.1016/j.aim.2019.106898}.
	
	\bibitem{Tao2003}
	Terence Tao.
	\newblock A sharp bilinear restrictions estimate for paraboloids.
	\newblock {\em Geom. Funct. Anal.}, 13(6):1359--1384, 2003.
	\newblock \url{https://doi.org/10.1007/s00039-003-0449-0}.
	
	\bibitem{Tao2009}
	Terence Tao.
	\newblock A pseudoconformal compactification of the nonlinear {S}chr\"odinger
	equation and applications.
	\newblock {\em New York J. Math.}, 15:265--282, 2009.
	\newblock \url{http://nyjm.albany.edu:8000/j/2009/15_265.html}.
	
	\bibitem{TVV1998}
	Terence Tao, Ana Vargas, and Luis Vega.
	\newblock A bilinear approach to the restriction and {K}akeya conjectures.
	\newblock {\em J. Amer. Math. Soc.}, 11(4):967--1000, 1998.
	\newblock \url{https://doi.org/10.1090/S0894-0347-98-00278-1}.
	
\end{thebibliography}
\end{document}